\documentclass[12pt]{amsart}
\usepackage{amsmath, amsthm, amssymb, hyperref, xypic}
\xyoption{all}
\usepackage{fullpage}
\usepackage{amscd}
\usepackage{amsfonts}
\usepackage{verbatim}
\usepackage{enumerate} 
\usepackage{verbatim}
\usepackage{color}	

\newcommand{\mb}{\mathbb}
\DeclareMathOperator{\Hom}{Hom}
\DeclareMathOperator{\HA}{H_A}
\DeclareMathOperator{\HAo}{H_{A^{op}}}
\DeclareMathOperator{\Hk}{H_k}
\DeclareMathOperator{\RHom}{RHom}

\DeclareMathOperator{\R}{R}

\DeclareMathOperator{\hdet}{hdet}
\DeclareMathOperator{\id}{id}
\DeclareMathOperator{\Tr}{Tr}
\newcommand{\wh}{\widehat}

\DeclareMathOperator{\lgr}{\!-gr}
\DeclareMathOperator{\lGr}{\!-Gr}

\newcommand{\Z}{\mathbb{Z}}
\newcommand{\C}{\mathcal{C}}
\newcommand{\D}{\mathcal{D}}

\newcommand{\mc}{\mathcal}
\newcommand{\mf}{\mathfrak}

\renewcommand{\1}{\mathbf{1}}

\newcommand{\gr}{\ensuremath{{}^\mathrm{gr}}}

\newcommand{\perf}{\mathrm{perf}}
\newcommand{\fm}{\mathfrak{m}}

\DeclareMathOperator{\Ext}{Ext} 
\DeclareMathOperator{\Aut}{Aut}
\newcommand{\Autz}{\Aut_{\mathbb Z}}
\newcommand{\Autw}{\Aut_{{\mathbb Z}^w}}
\DeclareMathOperator{\bfl}{\mathfrak{l}} 

\DeclareMathOperator{\Tor}{Tor} 
 
\newcommand{\wt}{\widetilde}

\numberwithin{equation}{section}

\theoremstyle{plain}
\newtheorem{theorem}[equation]{Theorem}
\newtheorem{corollary}[equation]{Corollary}
\newtheorem{lemma}[equation]{Lemma}
\newtheorem{proposition}[equation]{Proposition}
\newtheorem{question}[equation]{Question}

\newtheorem{conjecture}[equation]{Conjecture}

\theoremstyle{definition}
\newtheorem{definition}[equation]{Definition}

\newtheorem{remark}[equation]{Remark}

\newtheorem{hypothesis}[equation]{Hypothesis}

\begin{document}

\title[Skew Calabi-Yau triangulated categories]
{Skew Calabi-Yau triangulated categories\\ and Frobenius Ext-algebras}

\author{Manuel Reyes}
\address{Department of Mathematics\\ Bowdoin College\\
8600 College Station\\ Brunswick, ME 04011-8486\\ USA}
\email{reyes@bowdoin.edu}

\author{Daniel Rogalski}
\address{UCSD\\ Department of Mathematics\\ 9500 Gilman Dr. \# 0112 \\
La Jolla, CA 92093-0112\\ USA}
\email{drogalsk@math.ucsd.edu}

\author{James J. Zhang}
\address{University of Washington\\ Department of Mathematics\\ Box 354350\\
Seattle, WA 98195-4350\\ USA}
\email{zhang@math.washington.edu}





\date{February 22, 2016}

\begin{abstract}
We investigate conditions that are sufficient to make the Ext-algebra 
of an object in a (triangulated) category into a Frobenius algebra,
and compute the corresponding Nakayama automorphism. As an application,
we prove the conjecture that $\hdet(\mu_A) = 1$ for any noetherian 
Artin-Schelter regular (hence skew Calabi-Yau) algebra $A$.
\end{abstract}

\maketitle

\section{Introduction}
\label{xxsec1}

Let $k$ be an algebraically closed field. The goal of this paper is to 
study some categories with nice duality theories that arise
from Artin-Schelter (or AS for short) Gorenstein $k$-algebras. 
By studying the Ext-algebras of objects in these categories we will 
obtain several interesting consequences for such Gorenstein algebras. 
This paper is a sequel to \cite{RRZ}.  In particular, 
we make significant progress on a conjecture in \cite{RRZ} about the 
homological determinant of a Nakayama automorphism (see 
Theorem~\ref{xxthm1.3} below).

An especially pleasing form of duality is encapsulated in the 
definition of a \emph{Calabi-Yau triangulated category}~\cite{Keller}. 
This is a Hom-finite $k$-linear triangulated category with a duality 
of the form 
\begin{equation}
\label{E1.0.1}\tag{E1.0.1}
\Hom(X,Y)^* \cong \Hom(Y, \Sigma^d X),
\end{equation}
where $(-)^*$ denotes the $k$-linear dual and $\Sigma$ is the translation 
functor. Of course, the theory of Serre duality is the model for this 
definition, and the bounded derived category of coherent sheaves over a 
smooth projective variety with trivial dualizing sheaf is an important 
example.  Other interesting examples include the bounded derived categories of finite-dimensional modules over 
\emph{Calabi-Yau algebras}~\cite{Ginzburg}.

We are especially interested in derived categories of graded modules 
for those graded algebras that are important in noncommutative 
algebraic geometry, such as AS regular or Gorenstein 
algebras~\cite{ArtinSchelter}, in which case the duality theory above 
does not capture all examples of interest.  For this reason we study more 
general dualities of the form 
\begin{equation}
\label{E1.0.2}\tag{E1.0.2}
\Hom(X,Y)^* \cong \Hom(Y, \Sigma^d T^{\bfl} \Phi(X)),
\end{equation}
where $T$ is an automorphism of a triangulated category (typically 
coming from a shift of grading on modules), $\Phi$ is another 
automorphism called the \emph{Nakayama functor}, and $\Sigma, T$, 
and $\Phi$ all commute. Under some additional technical conditions, 
we call a $k$-linear Hom-finite triangulated category satisfying a duality 
of this form a \emph{skew Calabi-Yau triangulated category} (see 
Section~\ref{xxsec2.5} for the precise definition).


It is well-known that the duality condition~\eqref{E1.0.1} endows each
Ext-algebra $E(X) = \bigoplus_{i=0}^d \Hom(X, \Sigma^i X)$ with the 
structure of a graded-symmetric Frobenius algebra. 
Our first theorem shows 
that this extends in some cases to our more general kind of duality.

\begin{theorem}[Theorem~\ref{xxthm2.7}]
\label{xxthm1.1}
If $\mc{C}$ is a skew Calabi-Yau triangulated category with Nakayama 
functor $\Phi$, and $X$ is an object such that there is an isomorphism 
$\phi: X \to \Phi(X)$, then the Ext-algebra $E(X) = \bigoplus_{i, j} 
\Hom(X, \Sigma^iT^j X)$ has the structure of a bigraded Frobenius algebra.  

Furthermore, there is an explicit formula for the  Nakayama automorphism 
$\mu$ of the Frobenius algebra $E(X)$ in terms of the given isomorphism 
$\phi$.
\end{theorem}

The main examples we have in mind are certain subcategories of derived categories of 
modules over Gorenstein rings.  Recall that an $\mb{N}$-graded 
$k$-algebra $A = \bigoplus_{i \geq 0} A_i$ is called \emph{AS Gorenstein} 
if it is connected ($A_0 = k$) and locally finite ($\dim_k A_i<\infty$ 
for all $i \geq 0$), $A$ has finite injective dimension $d$ as an 
$A$-module, and $\Ext^d_A(k, A) \cong k(\bfl)$ (as graded modules), while $\Ext^i_A(k, A) = 0$ if $i \neq d$.  
Here, $k = A/A_{\geq 1}$ 
is the trivial module and $k(\bfl)$ is its graded shift by 
$\bfl\in {\mathbb Z}$.  If 
$A$ has finite global dimension in addition, then $A$ is called \emph{AS regular}.   

Let $A$ be noetherian AS Gorenstein, and let $\D_{\epsilon}(A)$ be the 
subcategory of the bounded derived category of finitely generated 
$\mb{Z}$-graded $A$-modules consisting of perfect complexes with 
finite-dimensional cohomology.    Let $\Autz(A)$ be the group of all graded algebra automorphisms
of $A$. The algebra $A$ has an associated \emph{Nakayama automorphism} $\mu_A \in \Autz(A)$ 
(see Definition~\ref{xxdef3.1}).  The operation  $M \mapsto {}^\mu M$ which twists the action on a graded left module 
$M$ by this automorphism induces a functor $\Phi \colon \D_{\epsilon}(A) \to \D_{\epsilon}(A)$.  
An object $X \in \D_{\epsilon}(A)$ is called \emph{$\Phi$-plain} if $\phi(X) \cong X$, and we let 
$\D^{pl}_{\epsilon}(A)$ be the thick subcategory of $\D_{\epsilon}(A)$ generated by $\Phi$-plain 
complexes.   The complex shift functor $\Sigma: X \mapsto X[1]$ and the functor $T$ induced by 
the graded shift of modules $M \mapsto M(1)$ restrict to this category $\D^{pl}_{\epsilon}(A)$.
We prove the following. 

\begin{proposition}[Proposition~\ref{xxpro3.3} and Theorem \ref{xxthm3.5}]
\label{xxpro1.2}
Let $A$ be noetherian AS Gorenstein.  Then 
$\D^{pl}_{\epsilon}(A)$, with its functors $\Sigma$, $T$, and $\Phi$ as above, 
is a skew Calabi-Yau triangulated category.
\end{proposition}
\noindent  In fact, we prove the proposition for a more general class of 
$\mb{N}$-graded algebras called \emph{generalized AS Gorenstein} 
(Definition~\ref{xxdef3.1}) which are locally finite but not necessarily connected.

If $A$ is noetherian (connected) AS Gorenstein, 
then each graded algebra automorphism of $A$ has a corresponding \emph{homological determinant}, 
which is a nonzero scalar, and where $\hdet: \Autz(A) \to k$ is multiplicative.   
The homological determinant is fundamental to the study of group 
actions on noncommutative graded algebras (see Section~\ref{xxsec4} for more details).   One of 
our motivating goals is to prove that the homological identity $\hdet(\mu_A) = 1$ holds in wide generality.
This identity has several significant applications; see \cite[Corollaries 0.5, 0.6, 0.7]{RRZ}.

In this paper, as an application of Theorem~\ref{xxthm1.1}, we prove 
\begin{theorem}[Theorem~\ref{xxthm5.3}]
\label{xxthm1.3}
If $A$ is noetherian connected AS Gorenstein and $\D_{\epsilon}(A) \neq 0$, 
then $\hdet(\mu_A) = 1$.
\end{theorem}
The homological identity $\hdet(\mu_A) = 1$ was proved in \cite[Theorem 6.3]{RRZ} in the special 
case that $A$ is noetherian Koszul AS regular, and was conjectured to hold for all noetherian AS Gorenstein algebras 
in \cite[Conjecture 6.4]{RRZ}.
The hypothesis $D_{\epsilon}(A) \neq 0$ 
of Theorem~\ref{xxthm1.3} seems very weak and we  conjecture that it holds 
automatically.  In any case, in Section~\ref{xxsec5} below we show that 
$D_{\epsilon}(A) \neq 0$ holds under very general conditions which cover most known AS 
Gorenstein algebras.   In particular, $D_{\epsilon}(A) \neq 0$ holds when 
$A$ is noetherian AS regular, and so the conjecture is proved for all such regular algebras
(Corollary~\ref{xxcor5.4}).  

As a second application of Theorem~\ref{xxthm1.1}, we recover and 
generalize a result of Berger and Marconnet \cite{BM}, as follows.  We recall that connected graded noetherian AS regular algebras $A$ are examples of \emph{skew Calabi-Yau algebras}
(see \cite[Definition 0.1]{RRZ}, \cite[Lemma 1.2]{RRZ}) and are \emph{Calabi-Yau algebras} when $\mu_A = 1$, 
a special case of particular interest.  

\begin{theorem}[Theorem~\ref{xxthm4.2}]
\label{xxthm1.4}
Let $A$ be connected graded noetherian AS regular of dimension $d$ with Nakayama 
automorphism $\mu_A$, where $A$ is generated in degree $1$, and let 
$E:=\bigoplus_i \Ext^i_A(k, k)$ be the Ext-algebra of $A$.  
Let $\mu_E$ denote the Nakayama automorphism of the Frobenius algebra $E$. 
Then identifying $(A_1)^*$ with $E^1 = \Ext^1_A(k, k)$, we have 
$\mu_E \vert_{E^1} = (-1)^{d+1} (\mu_A \vert_{A_1})^*$.
In particular, $\mu_A = 1$  (that is, $A$ is a Calabi-Yau algebra) if and only if $E$ is a graded-symmetric 
Frobenius algebra.
\end{theorem}
\noindent 
 
Portions of this theorem have been proved in the case of (not necessarily noetherian) $N$-Koszul 
algebras; see~ \cite[Theorem~6.3]{BM} for the description of $\mu_E|_{E_1}$ 
and~\cite[Proposition~3.3]{HVOZ} for the characterization of when $A$ is Calabi-Yau.
But many regular algebras of dimension 4 and higher are not 
$N$-Koszul.  In fact, Theorem~\ref{xxthm4.2} describes $\mu_E$ entirely, 
not just in degree $1$, provided one can calculate an explicit isomorphism of 
complexes between the minimal free resolution $X$ of $k$ and its twist  $\Phi(X)$ by 
$\mu_A$.   In addition, in Theorem~\ref{xxthm4.2} we give a more general result
which applies to certain not necessarily connected algebras.
The theorem also has further applications in noncommutative 
invariant theory; see Section~\ref{xxsec4.2}.

\subsection*{Acknowledgments}
This material is based upon work supported by the National Science
Foundation under Grant No.\ 0932078 000, while the authors were in
residence at the Mathematical Science Research Institute (MSRI) in
Berkeley, California, for the workshop titled ``Noncommutative Algebraic
Geometry and Representation Theory'' during the year of 2013.
Reyes, Rogalski, and Zhang were also supported by the respective National Science
Foundation grants DMS-1407152, DMS-1201572, and DMS-0855743 \& DMS-1402863.
Reyes was supported by an AMS-Simons Travel Grant.

\section{Calabi-Yau categories and Serre functors}
\label{xxsec2}
\subsection{(Bi)-graded categories}
\label{xxsec2.1}
This section follows Van den Bergh's treatment of Serre duality and 
Calabi-Yau triangulated categories given in~\cite[Appendix~A]{Bo}.  
Another standard reference on the subject is~\cite{Keller}.

Throughout the paper we work over an algebraically closed field $k$. 
Recall that a category $\C$ is \emph{$k$-linear} if each $\Hom_{\C}(X, Y)$ is a $k$-vector space, 
composition of morphisms is $k$-bilinear, and $\C$ has finite biproducts.  
For convenience, we assume that all categories in this paper are $k$-linear, although some definitions make sense 
in greater generality.   

A \emph{graded category} is a pair $(\C, \Sigma)$, where $\C$ is a $k$-linear category, and $\Sigma$ 
is a $k$-linear automorphism of $\C$. 
Given such a category with objects $X,Y \in \C$, we define 
\[
\Hom^i(X,Y) = \Hom(X,\Sigma^i Y),
\]
and then define the graded Hom-set
\[
\Hom\gr(X,Y) = \bigoplus_{i \in \mb{Z}} \Hom^i(X,Y).
\]
Graded composition is given, for $g \in \Hom^j(Y,Z)$ and $f \in \Hom^i(X,Y)$,
by $g * f = \Sigma^i(g) \circ f$.  It is easy to check that graded composition is associative.   
This allows one to define the category $\C\gr$ with the same objects as $\C$, but using graded Hom-sets and 
composition as above. 
In particular, for any object $X \in \C$, this 
endows $E(X) :=\Hom\gr(X,X) = \bigoplus_{i \in \mb{Z}} \Hom(X, \Sigma^i X)$ 
with the structure of a $\mb{Z}$-graded $k$-algebra.

A \emph{graded functor} between graded categories  is a pair
$(U, \eta^U) \colon (\C, \Sigma_\C) \to (\D, \Sigma_\D)$, 
where $U \colon \C \to \D$ is a $k$-linear functor and 
$\eta^U \colon U \circ \Sigma_\C \to \Sigma_\D \circ U$ is a natural 
isomorphism.
To simplify notation, we write $\eta^U = \eta$ and write 
both $\Sigma_{\C}$ and $\Sigma_{\D}$ as $\Sigma$, when there is no chance 
of confusion.  Given a graded functor 
$(U, \eta)$, for every integer $i \geq 1$ there is an associated natural 
isomorphism $\eta^i: U \circ \Sigma^i \to \Sigma^i \circ U$,
defined on an object $X$ to be the composite
\[
\eta^i_X \colon U \circ \Sigma^i(X) \overset{\eta_{\Sigma^{i-1}(X)}}{\to} 
\Sigma \circ U \circ \Sigma^{i-1}(X) \to \cdots
\to \Sigma^{i-1} \circ U \circ \Sigma(X) \overset{\Sigma^{i-1}(\eta_X)}{\to} 
\Sigma^i \circ U(X).
\]
Of course, the use of the symbol $\eta^i$ is a (harmless) abuse of 
notation.  We also get an induced natural isomorphism 
$\eta^{-1}: U \circ \Sigma^{-1} \to \Sigma^{-1} \circ U$ defined on an 
object $X$ by $(\eta^{-1})_X = \Sigma^{-1}(\eta_{\Sigma^{-1}(X)}^{-1})$.  
Then the powers $\eta^{-i} = (\eta^{-1})^i$ are determined as above for $i \geq 1$.  
Finally, let $\eta^0: U \to U$ be the identity natural isomorphism by 
convention.
It is routine to see that with these definitions, the following formula holds 
for all $i, j \in \mb{Z}$:
\begin{equation}
\label{E2.0.1}\tag{E2.0.1}
\eta^{i+j}_X = \Sigma^i(\eta^j_X) \circ \eta^i_{\Sigma^j(X)}.
\end{equation}
Now using \eqref{E2.0.1}, 
one may easily check that the graded functor $U$ induces a functor 
$U\gr \colon \C\gr \to \D\gr$, which agrees with $U$ on objects and is 
defined on a homogeneous element
$f \in \Hom_\C^i(X,Y) = \Hom_\C(X,\Sigma^i Y)$ by
\[
U\gr(f) = \eta^i_Y \circ U(f).
\]

Suppose that $(U, \eta^U), (V, \eta^V) \colon (\C,\Sigma_\C) \to 
(\D, \Sigma_D)$ are graded functors. A \emph{morphism of graded 
functors} is a natural transformation $h \colon U \to V$ such that, 
for every object $X$ of $\C$, we have  $\Sigma_\D(h_X) \circ \eta^U_X = 
\eta^V_X \circ h_{\Sigma_\C X}$. Of course, $h$ is called an 
\emph{isomorphism of graded functors} if it is a natural isomorphism
(whose inverse will necessarily be graded), 
and in this case we write $(U, \eta^U) \cong (V, \eta^V)$.

One can also compose graded functors in the obvious way.  If 
$(U, \eta^U) \colon (\C_1, \Sigma_1) \to (\C_2, \Sigma_2)$ and 
$(V, \eta^V) \colon (\C_2, \Sigma_2) \to (\C_3, \Sigma_3)$ are 
graded functors, then the composite $(V, \eta^V) \circ (U, \eta^U)$ 
is defined to be $(V \circ U, \eta^{V \circ U}) \colon (C_1, \Sigma_1) 
\to (\C_3, \Sigma_3)$, where $\eta^{V \circ U}$ is defined for an object 
$X \in \C_1$ by $\eta^{V \circ U}_X = \eta^V_{U(X)} \circ V(\eta^U_X)$.

In this paper, we also need a bigraded (or multi-graded) version of 
the above constructions, which is a routine generalization. We only 
define the bigraded case. The extension to the case of more than two 
automorphisms to get multi-graded categories is similar, and is left 
to the reader. Suppose that the $k$-linear category $\C$ has two 
automorphisms $\Sigma$ and $T$, which commute in the sense that 
$\Sigma T = T \Sigma$ as functors.  We say that $(\C, \Sigma, T)$ is 
a \emph{bigraded} category. We can define bigraded Hom sets 
$\Hom\gr(X, Y) = \bigoplus_{i,j} \Hom^{i,j}(X, Y)$, where 
$\Hom^{i,j}(X,Y) = \Hom(X, \Sigma^i T^j Y)$.  Composition is done in 
the obvious way:  if $g \in \Hom^{k,l}(Y,Z)$ and $f \in \Hom^{i,j}(X, Y)$, 
then 
$$g * f = \Sigma^iT^j(g) \circ f,$$
which is in $\Hom(X, \Sigma^i T^j \Sigma^k T^l Z) 
= \Hom(X, \Sigma^{i+k}T^{j + l} Z) = \Hom^{i+k, j + l}(X, Z).$
Note that since we assume that $\Sigma$ and $T$ strictly commute, rather 
than $\Sigma T$ and $T \Sigma$ only being naturally isomorphic, we 
silently commute $\Sigma$ and $T$ in formulas like this.   Similarly as 
in the singly graded case, it is routine to check that graded composition is associative.
Thus we may define a category $\C\gr$ with the same objects 
as $\C$, but with morphisms given by  bigraded Hom-sets and composition 
as above.  
In particular, the endomorphism ring $\Hom\gr(X, X)$ is a 
bigraded $k$-algebra, for any object $X$.

A \emph{bigraded functor} between graded categories 
$$(U, \eta^U, \theta^U) \colon (\C, \Sigma_\C, T_{\C}) \to 
(\D, \Sigma_\D, T_{\D})$$ 
is a $k$-linear functor
$U \colon \C \to \D$ together with natural isomorphisms 
\[
\eta^U \colon U \circ \Sigma_\C
\to \Sigma_\D \circ U\ \text{and}\ \theta^U:  
U \circ T_{\C} \to T_{\D} \circ U,
\]
satisfying the following compatibility condition:
for every object $X \in \C$, there is a commutative diagram
\begin{equation}
\label{E2.0.2}\tag{E2.0.2}
\xymatrix{
U \Sigma_{\C} T_{\C} X \ar[r] \ar[d]^{\Sigma_{\D}(\theta_X) 
\circ \eta_{T_{\C}(X)}} 
& U T_{\C} \Sigma_{\C} X \ar[d]^{T_{\D}(\eta_X) \circ 
\theta_{\Sigma_{\C}(X)}} \\
\Sigma_{\D} T_{\D} U X  \ar[r] & T_{\D} \Sigma_{\D} U X}
\end{equation}
where the horizontal arrows are induced from the respective equalities 
$\Sigma_{\C} T_{\C} = T_{\C} \Sigma_{\C}$ and $\Sigma_{\D} T_{\D} = 
T_{\D} \Sigma_{\D}$. Similarly as above, for simplicity from now on we 
will omit the superscripts on $\eta$ and $\theta$ and the subscripts 
on $\Sigma$ and $T$.

Given a bigraded functor $(U, \eta, \theta)$, for every pair of integers 
$i,j \in \mb{Z}$ there is an associated natural transformation 
$\eta^i \theta^j : U \Sigma^i T^j  \to 
\Sigma^i T^j U$, defined on an object $X$ as 
$\Sigma^i(\theta^j_X) \circ \eta^i_{T^j(X)}$, 
where the powers of $\eta$ and $\theta$ are defined using the 
conventions introduced earlier.
Now a similarly routine proof as in the singly graded case (but using also the 
compatibility condition \eqref{E2.0.2}) shows that any bigraded functor $(U, \eta, \theta)$ induces a functor $U\gr \colon 
\C\gr \to \D\gr$, which agrees with $U$ on objects and is 
defined on a homogeneous element
$f \in \Hom_\C^{i,j}(X,Y) = \Hom_\C(X,\Sigma^i T^j Y)$ by
\[
U\gr(f) = (\eta^i \theta^j)_Y \circ U(f).
\]

\subsection{Triangulated categories and exact functors}
\label{xxsec2.2}
The graded categories $(\mc{C}, \Sigma)$ we are interested in will 
typically be triangulated, and when this is the case $\Sigma$ will 
stand for the translation functor $X \mapsto X[1]$ unless otherwise indicated.
Let $\C$ and $\D$ be triangulated categories with translation 
functors $\Sigma_{\C}$ and $\Sigma_{\D}$ respectively. 
A graded functor $(U, \eta):\C\to \D$ is called \emph{exact} if $U$ maps
exact triangles in $\C$ to exact triangles in $\D$, that is, for
any exact triangle
$$X\xrightarrow{\quad f\quad } Y\xrightarrow{\quad g\quad } 
Z\xrightarrow{\quad h\quad } \Sigma_{\C} X,$$
the triangle in $\D$
$$U(X)\xrightarrow{U(f)} U(Y)\xrightarrow{U(g)} U(Z)
\xrightarrow{\eta_X\circ U(h)} \Sigma_{\D} U(X)$$
is exact.

\begin{remark} 
\label{xxrem2.1}
Let $\C$ be a triangulated category, considered as a graded category 
$(\C, \Sigma)$ where $\Sigma$ is the translation functor.
Then 
$$\Sigma^{+}:=(\Sigma, \eta_X=\operatorname{id}_{\Sigma^2(X)})$$
is a graded endofunctor of $(\C, \Sigma)$; however,  $\Sigma^{+}$ need not be exact.
On the other hand, 
$$\Sigma^{-}:=(\Sigma, \eta_X=-\operatorname{id}_{\Sigma^2(X)})$$ 
is a graded and exact functor of 
the triangulated category $(\C, \Sigma)$ \cite[Example A.3.2]{Bo}.
Thus $\Sigma^{-}$ is the natural way to make $\Sigma$ into a graded 
functor in the triangulated setting.
\end{remark}

\subsection{Categories with Serre functors}
\label{xxsec2.3}
Serre functors were introduced by Bondal and Kapranov 
in~\cite{BondalKapranov}. A thorough treatment of Serre functors is 
given by Reiten and Van den Bergh in~\cite[\S I.1]{ReitenVanDenBergh}, 
and from the perspective of graded categories in Van den Bergh's 
appendix~\cite[Appendix~A]{Bo}. Let $\Hom$ denote
$\Hom_{\C}$ for an abstract category $\C$.  The $k$-linear category $\C$ is 
called \emph{Hom-finite} if $\Hom(X, Y)$ is a finite-dimensional $k$-space 
for all $X, Y \in \C$.

\begin{definition}
\label{xxdef2.2}
Given a $k$-linear Hom-finite category $\C$, an autoequivalence 
$$F \colon \C \to \C$$
is called a \emph{{\rm{(}}right{\rm{)}} Serre functor} if for all 
$X,Y \in \C$ there are isomorphisms
\[
\alpha_{X,Y} \colon \Hom(X,Y) \to \Hom(Y, FX)^*
\]
that are natural in both $X$ and $Y$. 
\end{definition}

For a given object $X \in \C$, setting $Y = X$ in the natural 
isomorphism above, there is a distinguished element 
$\Tr_X = \alpha_{X,X}(\id_X) \in \Hom(X,FX)^*$.
This is the \emph{trace} map
\[
\Tr_X \colon \Hom(X, FX) \to k.
\]
The trace map provides a nondegenerate bilinear pairing
\[
\Tr_X(- \circ -) \colon \Hom(Y, FX) \times \Hom(X, Y) \to k,
\]
which satisfies the following identity for $(g,f) \in \Hom(Y, FX) 
\times \Hom(X, Y)$:
\begin{equation}
\label{E2.2.1}\tag{E2.2.1}
\Tr_X(g \circ f) = \Tr_Y(F(f) \circ g),
\end{equation}
see \cite[p. 29 (3)]{Bo}.
Conversely, the action of the functor $F$ on objects along with the 
trace maps $\Tr_X$ are enough to determine the functor $F$ and the 
isomorphisms $\alpha_{X,Y}$  
\cite[Prop.~I.1.4]{ReitenVanDenBergh}.

Suppose that $F$ is a Serre functor on a Hom-finite, $k$-linear 
graded category $(\C, S)$.  
We claim that there is an induced commutation rule for the 
functors $F$ and $S$.   Fix the isomorphisms $\alpha_{X, Y}$ which give 
the Serre duality.  Also, fix a nonzero scalar $\epsilon \in k^{\times} := k \setminus \{ 0 \}$.
Consider the diagram
\[  \xymatrix@C=30pt{
\Hom(X, Y) \ar[d]^{S} \ar[r]^{\alpha_{X,Y}} 
& \Hom(Y, FX)^* \ar[r]^{(S^*)^{-1}} 
& \Hom(S Y, S F X)^* \ar@{-->}[ld]^{\beta_{X,Y}} \\
\Hom(S X, SY) \ar[r]^{\alpha_{S X, S Y}} 
& \Hom(S Y, F S X)^* & }
\]
in which all solid arrows are isomorphisms.  Thus for every pair of 
objects there is an induced isomorphism $\beta_{X,Y}: 
\Hom(S Y, S F X)^* \to \Hom(S Y, F S X)^*$ such that the diagram 
commutes up to the scalar $\epsilon$, in other words such that 
\[
\beta_{X, Y} \circ (S^*)^{-1} \circ \alpha_{X, Y} = 
\epsilon (\alpha_{SX, SY} \circ S).
\]
The isomorphisms $\beta_{X, Y}$ are clearly also natural in $X$ and $Y$.
Fixing $X$ and varying $Y$, we get an isomorphism of functors 
$\Hom(-,  SF X)^* \to \Hom(-, F S X)^*$, and thus by the 
Yoneda embedding, an isomorphism of objects 
$\eta_X: F S X \to S F X$.  
These $\eta_X$ are natural in $X$ 
and so define a natural isomorphism 
$\eta = \eta_{\epsilon}(S): F S \to S F$.  
Note that $\beta_{X, Y}$ is then given by the $k$-linear dual 
of the map $f \mapsto \eta_X \circ f$ for $f \in \Hom( SY, F SX)$. One 
can also give $\eta$ in terms of the trace functions, as is done in 
Van den Bergh's appendix to \cite{Bo}. Namely, specializing the diagram 
above to the case $X = Y$, we see that $\eta = \eta_{\epsilon}(S)$ 
satisfies the formula \cite[p. 30 (4)]{Bo}
\begin{equation}
\label{E2.2.2}\tag{E2.2.2}
\Tr_X(S^{-1}(\eta_X \circ f)) =  \epsilon \Tr_{S X}(f),
\end{equation}
for all objects $X$ and
morphisms $f \colon S X \to F S X$ of $\C$.  One can check that the 
collection of trace identities in \eqref{E2.2.2} for all objects 
$X$ also uniquely determines the natural isomorphism $\eta$.

\begin{remark}
\label{xxrem2.3}
Suppose that $\C$ is a Hom-finite $k$-linear triangulated category 
with translation functor $\Sigma$, such that $\C$ has a Serre functor 
$F$.
\begin{enumerate}
\item We would like to also require that the natural transformation 
$\eta_{\epsilon}(\Sigma): F \Sigma \to \Sigma F$ makes $F$ into an 
exact functor of $(\C, \Sigma)$. This is always the case when one 
chooses $\epsilon = -1$ \cite[Theorem A.4.4]{Bo}. Thus $\eta_{-1}(\Sigma)$ 
is the natural way of commuting $F$ and $\Sigma$ in a triangulated category.
\item 
On the other hand, for many other isomorphisms $T$ of $\C$, it is most 
natural to take $\epsilon = 1$ and consider 
$\eta_1(T): FT \to TF$.  This will be the case, for example, when we consider 
a subcategory $\C$ of the bounded derived category of 
complexes of graded modules over an algebra $A$ with a $\mb{Z}$-grading, 
and $T$ is induced by the shift of grading on a module.  
\end{enumerate}
\end{remark}

We would also like to consider the bigraded case, for which the following 
lemma will be useful.

\begin{lemma}
\label{xxlem2.4}
Let $\C$ be a Hom-finite $k$-linear category with Serre functor $F$ and fixed 
isomorphisms $\alpha_{X, Y}$ implementing the Serre duality.  Let 
$S_1, S_2$ be isomorphisms of $\C$ which commute, and consider the 
bigraded category $(\C, S_1, S_2)$.  Let $\epsilon_1, \epsilon_2 
\in k^\times$ and define $\eta=\eta_{\epsilon_1}(S_1)$, 
$\theta = \eta_{\epsilon_2}(S_2)$.  
\begin{enumerate}
\item 
$\rho = \eta_{\epsilon_1 \epsilon_2}(S_1 \circ S_2)$ 
satisfies $\rho_X = S_1(\theta_X)\circ \eta_{S_2 X}$
for all objects $X$.
\item  
$(F, \eta, \theta)$ is a bigraded functor of $(\C, S_1, S_2)$.
\item  
For all objects $X$ and $f \in \Hom(S_1^i S_2^j X, F S_1^i S_2^j X)$, 
we have 
\[
 \Tr_X(S_1^{-i} S_2^{-j}( (\eta^i \theta^j)_X \circ f)) = 
\epsilon_1^i \epsilon_2^j \Tr_{S_1^i S_2^j X}(f).
\]
\end{enumerate}
\end{lemma}

\begin{proof}
(1)  This follows from a routine diagram chase, 
using the definition of the natural transformations $\eta_{\epsilon}(-)$.

(2) Since $S_1 S_2 = S_2 S_1$ by assumption, applying part (1) to both 
$S_1 S_2$ and $S_2 S_1$ shows that the condition \eqref{E2.0.2} holds.

(3)  Recalling the definition of the natural isomorphism $\eta^i \theta^j$ from Section~\ref{xxsec2.1}, this follows easily from (1), induction, and \eqref{E2.2.2}.
\end{proof}

The Serre functor $F$ is known to be unique up to natural isomorphism 
of functors \cite[Lemma I.1.3]{ReitenVanDenBergh}. Once $F$ is fixed, there is some choice in the isomorphisms 
$\alpha_{X, Y}$ which implement the Serre duality.  The natural 
transformation $\eta_{\epsilon}(S)$ may depend on this choice and so 
is uniquely determined only once $F$ and the $\alpha_{X, Y}$ are fixed.
To conclude this section, we discuss in more detail what effect changing the $\alpha_{X, Y}$
has on $\eta_{\epsilon}(S)$.

Fix the Serre functor $F$, and consider two families of isomorphisms 
satisfying Definition~\ref{xxdef2.2}, say $\{ \alpha_{X, Y} \}$ and 
$\{ \alpha'_{X,Y} \}$.   Then for all pairs $X, Y$, we get an isomorphism
\[
\gamma_{X, Y} = \alpha'_{X, Y} \circ \alpha_{X, Y}^{-1} : 
\Hom(Y, FX)^* \to \Hom(Y, FX)^*.
\]
Varying $Y$, by Yoneda's lemma this gives a natural isomorphism 
$\Psi: F \to F$ such that $\gamma_{X, Y}$ is the $k$-linear dual of 
$f \mapsto \Psi_X \circ f$.  Conversely, given a natural isomorphism 
$\Psi: F \to F$ and a family of isomorphisms $\{ \alpha_{X, Y} \}$ satisfying 
Definition~\ref{xxdef2.2}, we can define $\gamma_{X, Y}$ as the $k$-linear dual 
of $f \mapsto \Psi_X \circ f$ and then define $\alpha'_{X, Y} = \gamma_{X, Y} \circ \alpha_{X, Y}$, giving 
another family $\{ \alpha'_{X, Y} \}$ of natural isomorphisms satisfying Definition~\ref{xxdef2.2}.
Thus the possible choices of $\{ \alpha_{X, Y} \}$ 
are in bijection with the group of natural isomorphisms from $F$ to itself, which is 
the group of units of the ring of endomorphisms of $F$ in the category of functors.  
Since $F$ is an autoequivalence of $\C$, it is routine to check that the ring of endomorphisms of $F$ 
is isomorphic to the ring of endomorphisms of the identity functor.   Moreover, the 
ring of  endomorphisms of the identity functor is called the \emph{center} of the category $\C$.
Thus the possible families of natural isomorphisms $\{ \alpha_{X, Y} \}$ satisfying Definition~\ref{xxdef2.2} are in bijective correspondence 
with the units of the center of $\C$.

Now suppose we have fixed $F$ and a family of natural isomorphisms $\{ \alpha_{X, Y} \}$, and consider 
another choice $\{ \alpha'_{X, Y} \}$ corresponding to $\Psi: F \to F$ as above.
Let $S$ be an isomorphism of $\C$, and define the natural transformation 
$\eta = \eta_{\epsilon}(S): SF \to FS$ for some fixed $\epsilon \in k^{\times}$, using the $\alpha_{X, Y}$.
Let $\eta' = \eta'_{\epsilon}(S): FS \to SF$ be defined using the $\alpha'_{X, Y}$ instead.  
 A straightforward diagram chase shows that for any object $X$,
\begin{equation}
\label{E2.4.1}\tag{E2.4.1}
\eta'_X = S(\Psi_X)^{-1} \circ \eta_X \circ \Psi_{SX}.
\end{equation}
If one recalls the definitions from Section~\ref{xxsec2.1}, another way to interpret this equation is to say that $(F, \eta)$ and $(F, \eta')$ are isomorphic as graded functors from the graded category
$(\C, S)$ to itself.  Conversely, if an isomorphism of graded functors $(F, \eta_{\epsilon}(S)) \cong 
(F, \rho)$ is given, then one may easily see that there is a choice of natural isomorphisms $\{ \alpha'_{X, Y} \}$ satisfying 
Definition~\ref{xxdef2.2} such that $\rho = \eta'_{\epsilon}(S)$.

\subsection{Frobenius endomorphism algebras}
\label{xxsec2.4}
 
Our main application of the general theory of the past few sections will 
be to a much more specialized setting, which we describe now.

\begin{hypothesis}
\label{xxhyp2.5} 
Let $\C$ be a Hom-finite $k$-linear triangulated category
with translation $\Sigma$.
\begin{enumerate}
\item
We assume that $\C$ has a Serre functor $F$ and that 
isomorphisms $\alpha_{X, Y}$ implementing the Serre duality as in 
Definition~\ref{xxdef2.2} are fixed.  
We fix the natural isomorphism $\eta = \eta_{-1}(\Sigma): F \Sigma \to 
\Sigma F$ which makes $F$ an exact functor of $(\C, \Sigma)$.
\item 
We assume that $T$ is an automorphism of $\C$ such that $T\Sigma = 
\Sigma T$ and that the identity natural isomorphism $\1_{T \Sigma}: 
T \Sigma \to \Sigma T$ also makes $(T, \1_{T \Sigma})$ an exact
functor of $(\C, \Sigma)$.  We fix the natural isomorphism 
$\theta = \eta_1(T): F T \to T F$, so that $(F, \eta, \theta)$ is a 
bigraded functor of $(\C, \Sigma, T)$, as in Lemma~\ref{xxlem2.4}. We say 
that $(C, \Sigma, T)$ is a \emph{$\mb{Z}$-graded triangulated category}.
\item 
We may work instead in the multi-graded setting and replace $T$ with 
a set $T_1, \dots, T_w$ of pairwise commuting automorphisms of $\C$, 
where each $T_i$ separately satisfies the properties of (2).  In this 
case we say that $(\C, \Sigma, T_1, \dots, T_w)$ is a 
\emph{${\mathbb Z}^w$-graded triangulated category}. All of our results 
extend easily to this multi-graded case, but we will state them only 
in the case of a single automorphism $T$ for simplicity.   Our results 
below can also be specialized to the case $w = 0$, where there are no automorphisms $T$, but 
we always assume that $w = 1$ unless otherwise stated.
\item 
We assume that there is an additional autoequivalence $\Phi$ of $C$
which also commutes with both $\Sigma$ and $T$, such that
$(\Phi, \1_{\Phi \Sigma})$ is an exact graded functor of $(\C, \Sigma)$.  
We assume that the Serre functor $F$ is of the form 
$F = (\Sigma)^d \circ T^{\bfl}  \circ \Phi$.
Since $\Sigma, T$, and $\Phi$ all commute, we also have 
$F\Sigma=\Sigma F$ and $FT=T F$ as functors. 
\item
Finally, we assume that the natural isomorphisms $\eta$ and $\theta$ are 
scalar multiples of the trivial commutations $\1_{F \Sigma} : F \Sigma \to \Sigma F$ 
and $\1_{F T}: F T \to T F$, in other words that 
$\eta = s \cdot \1_{F \Sigma}$ and $\theta = t \cdot \1_{FT}$ 
for some $s, t \in k^\times$. 
\end{enumerate}
\end{hypothesis}

In Section~\ref{xxsec3}, we will give many examples satisfying the 
hypothesis above, which arise from considering subcategories of the 
bounded derived category of graded modules over an AS Gorenstein algebra 
$A$ of dimension $d$.  In these examples $\Sigma$ is the shift of complexes, $T$ is 
induced from the shift of grading on modules, and $\Phi$ is induced 
from twisting modules by some automorphism of $A$.  Moreover, 
in these examples we will calculate that $s = (-1)^d$ and $t =1$.

The main goal of the rest of this section is to show that in some cases the graded endomorphism 
algebra of an object in a category satisfying Hypothesis~\ref{xxhyp2.5} is Frobenius, 
and to give a formula for its Nakayama automorphism.
We begin by proving a bigraded version of the \emph{graded Serre duality} 
result established by Van den Bergh in \cite[Proposition A.4.3]{Bo}.  
As usual, a multi-graded version also holds.  In the remaining results in this section, 
under Hypothesis~\ref{xxhyp2.5} we freely commute the isomorphisms $\Sigma, T$, 
and $\Phi$ in formulas and proofs; we indicate a natural isomorphism 
only when it is not necessarily the trivial commutation.  For example, 
to properly interpret the expression $\Tr_Y(\Phi(f) * g)$ in part (2) of the next result, 
one should identify $\Sigma^{d-i} T^{{\bfl}-j} \Phi \Sigma^i T^j$ with 
$\Sigma^d T^{\bfl} \Phi = F$.

\begin{lemma}
\label{xxlem2.6}
Assume Hypothesis~\ref{xxhyp2.5} and its notation.  Consider the 
bigraded category $(\C, \Sigma, T)$ and its associated category 
$\C^{gr}$ with bigraded Hom sets. 
\begin{enumerate}
\item 
Given objects $X, Y \in \C$ and homogeneous elements 
$f \in \Hom^{i, j}(X,Y)$ and $g \in \Hom^{-i, -j}(Y, FX)$, one has
\begin{equation}
\label{E2.6.1}\tag{E2.6.1}
\Tr_X(g * f) = (-1)^i \Tr_Y(F\gr(f) * g).
\end{equation}
\item 
The formula of part (1) can be reinterpreted as follows:
if $f \in \Hom^{i,j}(X,Y)$ and $g \in \Hom^{d-i, {\bfl}-j}(Y,\Phi(X))$, then
\[
\Tr_X(g * f) = (-s)^i t^j \Tr_Y(\Phi(f) * g).
\]
\end{enumerate}
\end{lemma}

\begin{proof}
(1) This is very similar to Van den Bergh's proof of \cite[Proposition A.4.3]{Bo}, but for completeness 
we include the details. Applying definitions  we have 
\begin{align*}
\Tr_Y(F\gr(f) * g) &= \Tr_Y(\Sigma^{-i}T^{-j}(F\gr(f)) \circ g)  \\
& = \Tr_Y(\Sigma^{-i}T^{-j}( (\eta^i \theta^j)_Y  \circ F(f)) \circ g) \\
 &= \Tr_Y(\Sigma^{-i}T^{-j}( (\eta^i \theta^j)_Y  \circ F(f)  
\circ \Sigma^iT^j(g)))  \\
&=  (-1)^i \Tr_{\Sigma^i T^j Y}(F(f) \circ \Sigma^iT^j(g))\ \ \ \;\; \quad 
\qquad\; \quad \text{by}\ \text{Lemma}\ \ref{xxlem2.4}(3)\\
&=  (-1)^i  \Tr_X(\Sigma^iT^j(g) \circ f)\ \ \ \qquad\qquad \quad 
\qquad\; \quad 
\text{by}\ \eqref{E2.2.1} \\ 
&= (-1)^i  \Tr_X(g * f),
\end{align*}
as we needed.

(2) We can reinterpret 
\[
g \in \Hom^{d-i, {\bfl}-j}(Y,\Phi(X)) = \Hom(Y,\Sigma^{d-i} T^{{\bfl}-j} \Phi(X))
\]
as the element 
\[
\wt{g} \in \Hom^{-i,-j}(Y, \Sigma^d T^{\bfl} \Phi(X)) = \Hom^{-i,-j}(Y, FX). 
\]
Thus $g$ and $\wt{g}$ are the same morphism, but as graded morphisms 
they have different degrees.

Now the graded Serre duality formula~\eqref{E2.6.1} for $f$ 
and $\wt{g}$ 
reads as follows:
\begin{align*}
\Tr_X(\Sigma^iT^j(\wt{g}) \circ f) &= \Tr_X(\wt{g} * f) \\
&= (-1)^i \Tr_Y(F\gr(f) * \wt{g}) & \mbox{by~\eqref{E2.6.1}}\\
&= (-1)^i \Tr_Y(\Sigma^{-i}T^{-j}(F\gr(f)) \circ \wt{g}) \\
&= (-1)^i 
\Tr_Y(\Sigma^{-i}T^{-j}([\eta^i\theta^j]_Y \circ \Sigma^d T^{\bfl} \Phi(f)) 
\circ \wt{g}) &  \\
&= (-s)^i t^j \Tr_Y(\Sigma^{d-i} T^{{\bfl}-j}\Phi(f) \circ \wt{g}).
& \mbox{by~Hyp.~\ref{xxhyp2.5}(5)}
\end{align*}
Viewing $g$ as an element of $\Hom^{d-i, {\bfl}-j}(Y,\Phi(X))$ 
again, this equation
translates to the desired formula:
\begin{align*}
\Tr_X(g * f) &= \Tr_X(\Sigma^iT^j(g) \circ f) \\
&= \Tr_X(\Sigma^iT^j(\tilde{g}) \circ f) \\
&= (-s)^i t^j \Tr_Y(\Sigma^{d-i}T^{{\bfl}-j} \Phi(f) \circ \wt{g}) \\
&= (-s)^i t^j \Tr_Y(\Phi(f) * g). \qedhere
\end{align*}
\end{proof}

We now show that the bigraded endomorphism algebra 
\[
\Hom\gr(X, X) = \bigoplus_{i,j}\Hom(X, \Sigma^i T^j X)
\]
is a Frobenius algebra in certain cases. Since we do not yet have any 
assumptions  to ensure that this algebra is finite dimensional, we rely on 
the infinite dimensional generalization of Frobenius algebra in~\cite{Jans}.
Following Jans, a $k$-algebra $A$ is \emph{Frobenius} if there is an 
associative nondegenerate bilinear pairing $(-,-) \colon A \times A \to k$, 
as well as an algebra automorphism $\mu$ of $A$
satisfying $(x,y) = (\mu(y), x)$ for all $x,y \in A$. This $\mu$ 
is called a \emph{Nakayama automorphism},
and in the classical case where $A$ is finite dimensional, the 
existence of $\mu$ follows directly from the existence of $(-,-)$.
For $\Hom\gr(X,X)$ to be Frobenius, one needs $\Phi$ to preserve 
$X$ in some sense.  Having $\Phi(X) = X$ would certainly be 
sufficient, but in practice one is much more likely to find 
$X \cong \Phi(X)$ than equality on the nose.  Under the latter assumption 
we can prove the following result.
For any homogeneous element $g$, the degree of $g$ is denoted by
$|g|$ or $\deg (g)$ below.

\begin{theorem}
\label{xxthm2.7}
Assume Hypothesis~\ref{xxhyp2.5} and its notation, so $(\C, \Sigma, T)$ 
is a bigraded category. 
Let $X \in \C$ be an object such that there exists an
isomorphism $\phi \colon X \to \Phi(X)$. Then 
\[
E(X) = \Hom\gr(X,X) = \bigoplus_{i,j} \Hom(X, \Sigma^i T^j X),
\]
with multiplication given by composition $*$ in $\C\gr$, 
is a bigraded Frobenius algebra.   An associative nondegenerate 
bilinear form is defined on homogeneous elements $f,g \in E(X)$ with 
degrees $|g| = (i,j)$ and $|f| = (d-i, {\bfl}-j)$ by
\[
(f, g) = \Tr_X(\phi * f * g)
= \Tr_X(\Sigma^d T^{\bfl}(\phi) \circ \Sigma^i T^j(f) \circ g), 
\]
and by $(f, g) = 0$ if $|f| + |g| \neq (d, \bfl)$.
The corresponding Nakayama automorphism is defined as follows:
for any $g \in \Hom^{i,j}(X, X)$,
$$
g \longmapsto (-s)^i t^j \, \phi^{-1} * \Phi(g) * \phi 
= (-s)^i t^j \, \Sigma^i T^j(\phi)^{-1} \circ \Phi(g) \circ \phi.
$$
\end{theorem}

\begin{proof}
To check associativity of the pairing $(-,-)$, it suffices to assume that
$a,b,c \in E(X)$ are homogeneous with $|a| + |b| + |c| = (d,{\bfl})$ and note 
that
\[
(a * b, c)  = \Tr_X(\phi * (a * b) * c) = 
\Tr_X(\phi * a * (b * c)) = (a, b * c),
\]
relying upon associativity of graded composition.

To see that the pairing is nondegenerate, it is again sufficient to 
work with homogeneous elements in $E(X)$. 
Suppose that $g \in \Hom^{i,j}(X,X)$ 
is such that $(g,-) \colon E(X) \to k$ is the zero functional; we 
will show that $g = 0$. In particular, for every $f \in \Hom^{d-i, {\bfl}-j}(X,X)$,
we have 
\[
0 = (g,f) = \Tr_X(\phi * g * f) = 
\Tr_X(\Sigma^{d-i}T^{{\bfl}-j}(\phi * g) \circ f),
\]
where $\Sigma^{d-i}T^{{\bfl}-j}(\phi * g) \in 
\Hom(\Sigma^{d-i} T^{{\bfl}-j}X,\Sigma^d T^{\bfl}\Phi X) 
= \Hom(\Sigma^{d-i}T^{{\bfl}-j} X, FX)$.
Because the pairing $\Tr(- \circ -) \colon 
\Hom(\Sigma^{d-i}T^{{\bfl}-j} X, FX) \times 
\Hom(X, \Sigma^{d-i}T^{{\bfl}-j}X) \to k$
is nondegenerate, it follows that $\Sigma^{d-i}T^{{\bfl}-j}(\phi * g) = 0$. 
As $\Sigma$ and $T$ are automorphisms of $\C$, we obtain 
$0 = \phi * g = \Sigma^iT^j(\phi) \circ g$. Now because $\phi$ is 
an isomorphism, we conclude that $g = 0$ as desired.

Finally, we compute the Nakayama automorphism, making use of the graded
Serre duality formula from Lemma~\ref{xxlem2.6}. Assume that
$f,g \in E(X)$ are homogeneous, again with $|g| = (i,j)$ and 
$|f| = (d-i, {\bfl}-j)$.
Then $g \in \Hom^{i,j}(X,X)$ and 
$$\phi * f = \Sigma^{d-i} T^{\bfl-j}(\phi) \circ f 
\in \Hom(X,\Sigma^{d-i}T^{{\bfl}-j} \Phi X) =
\Hom^{-i,-j}(X,FX).$$  
Furthermore, as $\phi$ is an isomorphism in $\C$ and
$\phi \in \Hom^0(X,\Phi X)$, it has an inverse in $\C\gr$ which is 
also given by $\phi^{-1} \in \Hom^0(\Phi X, X)$.   So we have 
\begin{align*}
(f,g) &= \Tr_X((\phi * f) * g) \\
&= (-s)^i t^j \Tr_X(\Phi(g) * (\phi * f)) & 
\mbox{by Lemma~\ref{xxlem2.6}(2)}\\
&= (-s)^i t^j \Tr_X( \phi * (\phi^{-1} * \Phi(g) * \phi) * f) \\
&= (-s)^i t^j (\phi^{-1} * \Phi(g) * \phi, f).
\end{align*}
Thus the homomorphism $g \mapsto (-s)^it^j \phi^{-1} * \Phi(g) * \phi$
(for $g \in \Hom^{i,j}(X, X)$)
satisfies the defining property of a Nakayama automorphism of $E(X)$.
\end{proof}

\subsection{Skew Calabi-Yau triangulated categories}
\label{xxsec2.5}


If $\C$ is a Hom-finite $k$-linear triangulated category with Serre 
functor $F$, then as noted in Remark~\ref{xxrem2.3} we can make $F$ into a exact graded functor of 
$(\C, \Sigma)$ using $\eta = \eta_{-1}(\Sigma)$, which satisfies 
the trace formula \eqref{E2.2.2} with $\epsilon = -1$.  Then $\C$ is called a 
\emph{Calabi-Yau triangulated category} if
$(F, \eta) \cong (\Sigma, -\1_{\Sigma^2})^d = 
(\Sigma^d, (-1)^d \1_{\Sigma^{d+1}})$ as graded functors \cite{Keller}.
Recalling the discussion in Section~\ref{xxsec2.3}, it is equivalent to say that 
$F = \Sigma^d$ is a Serre functor, and there is a choice of the 
isomorphisms $\{ \alpha_{X, Y} \}$ implementing the Serre duality such that 
$\eta_{-1}(\Sigma) = (-1)^d \1_{\Sigma^{d+1}}$. 
Then Hypothesis~\ref{xxhyp2.5} applies in this case with 
$w = 0$, $\Phi$ the identity functor,  and with $s = (-1)^d$.    Theorem~\ref{xxthm2.7}
applies to every object $X$ of $\C$ with $\phi = \id_X$, so that 
the Ext-algebra $E(X)$ is always a graded-symmetric Frobenius 
algebra (that is, its Nakayama automorphism maps $g$ to $(-1)^{|g|(d+1)}
g$). These results are well known in the context 
of Calabi-Yau triangulated categories; 
see~\cite[Proposition~A.5.2]{Bo} and~\cite[Proposition~2.2]{Keller}.

Recall that a full subcategory $\D$ of a triangulated category $(\C, \Sigma)$ is called 
a \emph{triangulated subcategory} if $\D$ is closed under $\Sigma$ and if for all exact triangles, 
if two objects in the triangle are in $\D$, so is the third.  The full triangulated subcategory $\D$ of $\C$
is called \emph{thick} if it contains all direct summands of each of its objects.  Given a class $\mc{S}$ of 
objects of $\C$, the triangulated (or thick) subcategory generated by $\mc{S}$ is the smallest 
triangulated (respectively, thick) subcategory of $\C$ containing $\mc{S}$.
These categories have the following more explicit description.  Let $\mc{S}_1$ be the class of objects of the form 
$\{ \Sigma^i(X) | i \in \mb{Z}, X \in \mc{S} \}$.  Define $\mc{S}_n$ inductively for $n \geq 2$ as the class of 
all objects $Y$ which occur in exact triangles $X \to Y \to Z \to \Sigma(X)$ with $X, Z \in \mc{S}_{n-1}$.  Then the 
full subcategory with objects in $\mc{D} = \bigcup_{n \geq 1} \mc{S}_n$
is the triangulated subcategory generated by $\mc{S}$, and the full subcategory of all direct summands of objects in $\mc{D}$ is the thick subcategory 
generated by $\mc{S}$ \cite[Section 3.3]{Kr}.

\begin{definition}
\label{xxdef2.8}
Let $(\C, \Sigma)$ be a $k$-linear triangulated category
and $(\Phi, \eta)$ an exact autoequivalence of $\C$.
\begin{enumerate}
\item
An object $X$ in $\C$ is called $\Phi$-\emph{plain} if $\Phi(X)\cong X$.
\item
Let $\Xi(\Phi)$ denote the class of $\Phi$-plain objects in $\C$, 
and let $\C^{pl}$ be the thick subcategory of $\C$ generated by $\Xi(\Phi)$.
We say $\Phi$ is \emph{plain} if $\C^{pl} = \C$.
\end{enumerate}
\end{definition}

It is obvious that the identity functor is plain.  In general, $\Sigma^{-} = (\Sigma, - \1_{\Sigma^2})$
is not plain.  Similarly, in the $\Z$-graded triangulated categories $(\C, \Sigma, T)$ we study in the next section, usually 
$\Sigma$ and $T$ are not plain (see the proof of Lemma~\ref{xxlem3.6}(1)).
Intuitively, a plain functor must be close to being the identity
functor.  We now propose the following definition of \emph{skew Calabi-Yau
category} in the $\Z$-graded setting. Note that a modified version of the 
following definition can easily be given in both the multi-graded and  
the ungraded cases.

\begin{definition}
\label{xxdef2.9}
Let $(\C, \Sigma, T)$ be a $\mathbb{Z}$-graded $k$-linear Hom-finite triangulated category satisfying Hypothesis~\ref{xxhyp2.5}, so 
that in particular $\C$ has a Serre functor  of the form $F = \Sigma^d \circ T^{\ell} \circ \Phi$.
\begin{enumerate}
\item
We say that $\C$ is \emph{{\rm{(}}$\Z$-graded{\rm{)}} $\Phi$-skew 
Calabi-Yau} if (for some choice of maps $\{ \alpha_{X, Y} \}$ satisfying Definition~\ref{xxdef2.2}) we have 
\begin{gather*}
\label{E2.9.1}\tag{E2.9.1}
(F, \eta_{-1}(\Sigma)) =
(\Sigma, -\1_{\Sigma^2})^d \circ (T^{\bfl}, 
\1_{\Sigma T^{\bfl}}) \circ (\Phi, \1_{\Sigma\Phi}) = (F,(-1)^d \1_{\Sigma F}) \ \ \ \text{and} \\
(F, \eta_{1}(T))  = 
(\Sigma, \1_{T \Sigma})^d \circ (T^{\bfl}, 
\1_{T^{\bfl +1}}) \circ (\Phi, \1_{T\Phi}) = (F,\1_{TF}),
\end{gather*}
as graded functors.  Equivalently, $s = (-1)^d$ and $t= 1$ in Hypothesis~\ref{xxhyp2.5}.
\item
If there is a {\bf plain} exact autoequivalence $\Phi$ such that 
\eqref{E2.9.1} holds,
then $\C$ is called \emph{{\rm{(}}$\Z$-graded{\rm{)}} 
skew  Calabi-Yau of dimension~$d$}. In this case, we say that $\Phi$ is a 
\emph{Nakayama functor} for $\C$ and $\bfl$ is the \emph{AS index.}
\end{enumerate}
\end{definition}


The reason that~\ref{xxdef2.9}(1) is not a sufficient definition for a
skew Calabi-Yau category is that there is no guarantee of uniqueness of
the data $(d,\bfl,\Phi)$ in the definition. The requirement that $\Phi$ is
plain leads to uniqueness in the main case of interest studied in the next section;
see Lemma~\ref{xxlem3.6}(2) below. Presumably there may be some
less stringent condition than plainness that could lead to uniqueness of
 $(d,\bfl,\Phi)$ for a $\Phi$-skew Calabi-Yau triangulated category.

The next lemma shows that there is a standard way of producing skew 
Calabi-Yau categories from $\Phi$-skew Calabi-Yau categories. The lemma
will be used in the next section. 

\begin{lemma}
\label{xxlem2.10} 
Let $\C$ be a $\Phi$-skew Calabi-Yau category.  
Then $(\C^{pl}, \Sigma, T)$ is a skew Calabi-Yau category. That is, $\Phi$ is plain on $\C^{pl}$.
\end{lemma}

\begin{proof} 
Since $T$ commutes with $\Phi$, the functor $T$ maps $\Xi(\Phi)$ to itself.  Since $(T, \1_{T\Sigma})$ is 
an exact functor of $(\C, \Sigma)$, using the explicit description of 
the thick subcategory generated by a class of objects given earlier in this section, one easily sees that $T$
restricts to the subcategory $\C^{pl}$.  Similarly, $T^{-1}$ restricts to $\C^{pl}$, so $T$ restricts to an automorphism of 
$\C^{pl}$, and a straightforward argument can be used to show that $T^{-1}$ is also exact (alternatively, one can
apply~\cite[Lemma~49]{Murfet}).
Similarly, $\Sigma$ restricts to an automorphism of $\C^{pl}$, and $\Phi$ restricts to an autoequivalence of $\C^{pl}$.  
Thus $F$ restricts to an autoequivalence of $\C^{pl}$, and clearly Definition~\ref{xxdef2.9}(1) still holds for the restricted category, 
that is,  $(C^{pl}, \Sigma, T)$  is a $\Phi$-skew Calabi-Yau category.  By the definition of $\C^{pl}$, 
$\Phi$ is plain when restricted to $\C^{pl}$. The assertion follows. 
\end{proof}

\section{Skew Calabi-Yau categories related to AS Gorenstein algebras}
\label{xxsec3}
In this section, we show that a certain subcategory of the derived 
category of graded modules over an AS Gorenstein algebra is a skew 
Calabi-Yau triangulated category.

Throughout this section, $A$ will be an algebra which is 
$\mb{Z}^w$-graded for some $w \geq 1$, and $|x| \in \mb{Z}^w$ will indicate 
the degree of a homogeneous element in $A$ or in a $\mb{Z}^w$-graded 
$A$-module.   Any such 
algebra $A$ is automatically also $\mb{Z}$-graded, using the homomorphism 
$\mb{Z}^w \to \mb{Z}$ which adds the coordinates, and we use 
$|| x ||$ to indicate the total degree of $x$ under this $\mb{Z}$-grading.  
We are interested primarily in algebras which are 
$\mb{N}$-graded with respect to the $|| \;\; ||$-grading.  
Recall that an $\mb{N}$-graded algebra 
$A = \bigoplus_{n \geq 0}  A_n$ is \emph{connected} if $A_0 = k$,  
and \emph{locally finite} if $\dim_k A_n < \infty$ for all $n \geq 0$.
Given an $\mb{N}$-graded algebra $A$, let $A \lGr$ be the category of
$\mb{Z}$-graded left $A$-modules and let $\Gamma = \Gamma_{\mf{m}_A}$ 
be the \emph{torsion functor}
$A \lGr \to A \lGr$ which is defined as follows:
\[
\Gamma(M) = \{x \in M | A_{\geq n} x = 0,\ \text{some}\ n \geq 1 \} 
= \lim_{n \to \infty} \Hom_A(A/A_{\geq n}, M).
\]
If $A$ is $\mb{Z}^w$-graded, then given a $\mb{Z}^w$-graded $A$-module $M$, the notation $M^*$ will indicate the 
\emph{graded dual}, that is $\bigoplus_{g \in \mb{Z}^w} \Hom_k(M_g, k)$.
If $M$ is an $(A, A)$-bimodule and $\sigma, \mu$ are automorphisms of 
$A$, then $^{\sigma} M ^{\mu}$ is the new bimodule with the same 
underlying vector space as $M$, but with left and right actions twisted 
by the indicated automorphisms, that is with $a * m * b = \sigma(a) m \mu(b)$ for $m \in M$, $a, b \in A$.  Similarly, 
we can define the twist $^\sigma M$ for a left module $M$.

As in \cite{RRZ}, we would like our results to apply to 
$\mb{N}$-graded algebras which are locally finite but 
not necessarily connected.   Following \cite[Definition 3.3]{RRZ}, 
we generalize the notion of AS Gorenstein to this setting as follows.

\begin{definition}
\label{xxdef3.1}
Let $A$ be a ${\mathbb Z}^w$-graded algebra, for some $w\geq 1$, such
that it is locally finite and ${\mathbb N}$-graded with respect to the
$\mid\mid \;\;\mid\mid$-grading. We say $A$ is a {\it generalized
AS Gorenstein algebra} if
\begin{enumerate}
\item
$A$ has injective dimension $d$.
\item
$A$ is noetherian and satisfies the $\chi$ condition
\cite[Definition 3.7]{AZ}, and the 
functor $\Gamma_{\fm_A}$ has finite cohomological dimension.
\item
There is an $A$-bimodule isomorphism 
$R^d\Gamma_{\fm_A}(A)^*\cong {^\mu A^1}(-{\bfl})$, for some 
$\bfl\in{\mathbb Z}^w$ 
(called the {\it AS index}) and for some graded algebra
automorphism $\mu$ of $A$ (called the {\it Nakayama automorphism}).  
\end{enumerate}
If, in addition, $A$ has finite global dimension, then $A$ is called a 
\emph{generalized AS regular algebra}.
\end{definition}
\noindent
An important consequence of the definition above is that a generalized 
AS Gorenstein algebra $A$ will have a rigid dualizing complex of the form 
$U := R^d\Gamma_{\fm_A}(A)^*[d] \cong {^\mu A^1}(-{\bfl})[d]$ 
(see the proof of \cite[Lemma 3.5]{RRZ}).

Next, we define some important triangulated categories associated to a 
generalized AS Gorenstein algebra $A$.
Let ${\mathcal D}(A)$ be the derived category of graded left $A$-modules, 
and let ${\mathcal D}^b_{f}(A)$ be the bounded derived category of finitely 
generated graded left $A$-modules, which can be viewed as a full triangulated 
subcategory of ${\mathcal D}(A)$.  Let ${\mathcal E}(A)$ be the full 
triangulated subcategory of 
${\mathcal D}^b_{f}(A)$ consisting of complexes with finite dimensional 
cohomologies.   Recall that a complex is \emph{perfect} if it is quasi-isomorphic 
to a bounded complex of finitely generated projective modules.  
Let ${\mathcal D}_{\perf}(A)$ be the full triangulated subcategory
of ${\mathcal D}(A)$ consisting of perfect complexes of graded left 
$A$-modules.
Finally, let ${\mathcal D}_{\epsilon}(A)$ be the full 
triangulated subcategory of ${\mathcal D}(A)$ consisting of objects
in both ${\mathcal D}_{\perf}(A)$ and ${\mathcal E}(A)$.
Note that if $A$ is generalized AS regular, 
then ${\mathcal D}^b_{f}(A)={\mathcal D}_{\perf}(A)$ and 
${\mathcal D}_{\epsilon}(A)={\mathcal E}(A)$.  We let $\Sigma$ be the 
translation functor, in other words the shift $X \mapsto X[1]$
of complexes,  in any of these triangulated categories.

Let $U$ be the rigid dualizing complex as above, and let $F$ be the functor 
$U\otimes_A^L -$. Let $V={^1 A^\mu}(\bfl)[-d]$. Then it is obvious that
$U\otimes^L_A V \cong V\otimes^L_A U\cong A$ as complexes of graded 
$A$-bimodules, so $G = V\otimes_A^L -$ is an inverse of the automorphism
$F$. Let $D = \RHom_A(-, U)$ be the duality functor ${\mathcal D}^b_f(A)\to 
{\mathcal D}^b_f(A^{op})$, which has inverse 
$D^{op} = \RHom_{A^{op}}(-,U): {\mathcal D}^b_f(A^{op}) \to 
{\mathcal D}^b_f(A)$ \cite[Proposition 1.3]{YeZ1}.  
Since $A$ is generalized AS Gorenstein, $\sum_i \dim_k \Ext^i_A(M,A)$ is
finite for any finite dimensional graded $A$-module $M$.  This implies that $D(X)\in {\mathcal E}(A^{op})$ whenever 
$X\in {\mathcal E}(A)$. Therefore 
the pair of functors $(D,D^{op})$ restrict to the subcategory
${\mathcal E}(A)$. 
The pair of functors $(D,D^{op})$ clearly restrict 
to ${\mathcal D}_{\perf}(A)$,  and thus also to ${\mathcal D}_{\epsilon}(A) 
= \mc{D}_{\perf}(A) \cap \mc{E}(A)$.   Similarly, $F$ and $G$ restrict to 
$\mc{D}_{\epsilon}(A)$.  

We can now prove that the category $\mc{D}_{\epsilon}(A)$ satisfies 
Serre duality.   

\begin{lemma}
\label{xxlem3.2} 
Let $A$ be a generalized AS Gorenstein algebra, and maintain 
the notation introduced above.
\begin{enumerate}
\item
\cite[Theorem 5.1 and 6.3]{VdB1}
$\R\Gamma_{\fm_A}(-)^*\cong D(-)$, as functors from 
${\mathcal D}^b_f(A)$ to ${\mathcal D}^b_f(A^{op})$.
\item
If $Y$ is in ${\mathcal E}(A)$, then $Y^*\cong D(Y)$.
\item
For any $X\in {\mathcal D}_{\perf}(A)$ and $Y\in {\mathcal E}(A)$, 
there is an isomorphism 
\begin{equation}
\label{E3.2.1}\tag{E3.2.1}
\alpha_{X,Y}:\quad 
\Hom_{{\mathcal D}(A)}(X,Y)\cong \Hom_{{\mathcal D}(A)}
(Y, F(X))^*, 
\end{equation}
and these isomorphisms are natural in $X$ and $Y$.
\item
The functor $F$ defined above is a Serre functor of the 
category ${\mathcal D}_{\epsilon}(A)$, and \eqref{E3.2.1} defines a 
Serre duality when restricted to ${\mathcal D}_{\epsilon}(A)$.
\end{enumerate}
\end{lemma}

\begin{proof} (1)  Note that the indicated theorems from \cite{VdB1} do apply, 
since Van den Bergh's theory on local duality works
for locally finite noetherian algebras satisfying $\chi$ and having
finite cohomological dimension (see the comments in \cite[Lemma 3.5]{RRZ}).  
Hence the assertion holds. 

(2) This follows from part (1) and the fact that 
$\R\Gamma_{\fm_A}(Y)=Y$ \cite[Lemma 4.4]{VdB1}.

(3) Since $X$ is a perfect complex, we have
$$\begin{aligned}
\RHom_A(X,Y)&\cong \RHom_A(X, F(G(Y)))=
\RHom_A(X, U\otimes_A^L G(Y))\\
&\cong \RHom_A(X,U)\otimes_A^L G(Y)
=D(X)\otimes_A^L G(Y).
\end{aligned}
$$
Since $Y\in {\mathcal E}(A)$, it follows from induction on $\sum_{n} 
\dim_k H^0(Y[n])$ that $G(Y)\in {\mathcal E}(A)$. 
By
part (2), $D(G(Y))=G(Y)^*=\RHom_k(G(Y),k)$. 
By the above computation, adjointness, definition, and duality,
$$\begin{aligned}
\RHom_A(X,Y)^*&\cong 
\RHom_k(D(X)\otimes_A^L G(Y),k)\\
&\cong \RHom_{A^{op}}(D(X), \RHom_k(G(Y),k))\\
&\cong \RHom_{A^{op}}(D(X), D(G(Y)))\\
&\cong \RHom_{A}(G(Y),X)\cong \RHom_{A}(Y,F(X)).
\end{aligned}
$$
This is equivalent to 
$$\RHom_A(X,Y)\cong \RHom_{A}(Y,F(X))^*.$$
The assertion follows from taking $H^0$ of the above. 

(4) Since $F=U\otimes^L_A-$ is an autoequivalence of ${\mathcal D}_{\epsilon}(A)$, the assertion follows from
part (3).
\end{proof}

Suppose that $A$ is $\mb{N}$-graded generalized AS Gorenstein.
The shift of grading operation on a module $M \mapsto M(1)$, where 
$M(1)$ is defined by $M(1)_i = M_{i+1}$, gives an automorphism 
of the category $A \lGr$, which extends in an obvious way to complexes and thus gives an automorphism $T$ of $\mc{D}(A)$.  
Similarly, there is an automorphism of $A \lGr$ defined on objects by $M \mapsto {}^{\mu} M$, where $\mu = \mu_A$ is 
the Nakayama automorphism, and  as the identity 
on morphisms (using that $M$ and ${}^{\mu} M$ have the same underlying $k$-space).   It is easy to see that this functor 
may be identified with ${}^{\mu} A^1 \otimes_A -$.  This functor extends to complexes and thus gives an automorphism $\Phi$ of $\D(A)$, where in fact  $\Phi = {}^{\mu} A^1 \otimes^L_A -$.  Both $T$ and $\Phi$ restrict to $\mc{D}_{\epsilon}(A)$.
By the definition of $F$ and  Lemma~\ref{xxlem3.2}, 
$F := {}^{\mu} A ^1 (-\bfl)[d] \otimes^L_A -$ is the Serre 
functor of $\mc{D}_{\epsilon}(A)$, and clearly $F$ may also be written 
in the form $F = \Sigma^d \circ T^{-{\bfl}} \circ \Phi$.   It is also obvious 
that $\Sigma, T$, and $\Phi$ pairwise commute.   

We will see next that $(\mc{D}_{\epsilon}(A), \Sigma, T)$ is a $\Phi$-skew Calabi-Yau category 
in the sense of Definition~\ref{xxdef2.9}.  
 
\begin{proposition}
\label{xxpro3.3}
Let $A$ be $\mb{N}$-graded generalized AS Gorenstein and consider the 
bigraded category $(\mc{D}_{\epsilon}(A), \Sigma, T)$ 
defined above, where $\mc{D}_{\epsilon}(A)$ has Serre functor 
$F = \Sigma^d \circ T^{-\bfl} \circ \Phi$.
Fix the particular isomorphisms 
\[
\alpha_{X, Y}: \Hom_{\mc{D}_{\epsilon}(A)}(X,Y) \cong 
\Hom_{\mc{D}_{\epsilon}(A)}(Y,F(X))^*
\]
implementing the Serre duality which are determined in the course 
of the proof of {\rm{ Lemma {\ref{xxlem3.2}(3)}}}.  Then 
\begin{enumerate}
\item[{\rm{(a)}}] 
$\eta = \eta_{-1}(\Sigma) = (-1)^{d} \1_{F\Sigma}: 
F \Sigma \to \Sigma F$,  and 
\item[{\rm{(b)}}]  
$\theta = \eta_1(T) = \1_{FT}: FT \to TF$. 
\end{enumerate}
In particular, $(\mc{D}_{\epsilon}(A), \Sigma, T)$  satisfies 
Hypothesis~\ref{xxhyp2.5}, with 
$s = (-1)^d$ and $t = 1$, or equivalently 
$(\mc{D}_{\epsilon}(A), \Sigma, T)$ is a $\Phi$-skew Calabi-Yau triangulated category.  
\end{proposition}

In order to prove Proposition~\ref{xxpro3.3}, we found no alternative to a direct verification using the precise form of the isomorphisms  $\alpha_{X, Y}$, and since these are defined by composing quite a few maps, 
the verification is rather long and technical.  The reader willing to take the proof on faith may 
wish to skip ahead to Remark~\ref{xxrem3.4} at this point.

Since the proof below analyzes  elements and morphisms 
concerning some Hom-sets, it is useful to recall some basic 
notation to be used in the proof. Recall that $A \lGr$ is the category 
of graded left $A$-modules with morphisms being the graded $A$-module
homomorphisms of degree 0. For any graded left $A$-modules
$M$ and $N$, let $\Hom_A(M,N)$ denote the graded space
$\bigoplus_{i\in {\mathbb Z}} \Hom_{A\lGr}(M,T^i N)$.  Thus 
the degree 0 component of $\Hom_A(M,N)$ is
$\Hom_A(M,N)_0=\Hom_{A\lGr}(M,N)$. If $M$
is finitely generated, then $\Hom_A(M,N)$, when forgetting
the grading, agrees with the ungraded Hom set of all $A$-module 
homomorphisms from $M$ to $N$. Let $X$ and $Y$ be two  
complexes of graded left $A$-modules. Let $\Hom_A(X,Y)$
be the complex of graded $k$-modules, defined by 
$\Hom_A(X,Y)^i=\prod_{k\in {\mathbb Z}} \Hom_A(X^k, Y^{i+k})$, 
with its standard differential. If $X$ is a perfect complex or if 
$Y$ is $K$-injective (or semi-injective), then
$$H^0 \Hom_A(X,Y)_0= \Hom_{\mc{D}(A)}(X,Y).$$
When $X$ and $Y$ are objects in a full triangulated subcategory $\mc{B}$ 
of $\mc{D}(A)$, then $\Hom_{\mc{B}}(X,Y)=\Hom_{\mc{D}(A)}(X,Y)$.
In the next proof we work with morphisms at the level 
of complexes (or differential graded modules).   Then by taking 
$H^0(-)_0$, we will then have morphisms at the level of derived
categories. 

\begin{proof}[Proof of Proposition \ref{xxpro3.3}]
For any perfect complexes $X, Y \in \mc{D}_{\epsilon}(A)$, let
$$D(X):=\Hom_A(X,U), \quad 
F(X):=U\otimes_A X, \quad {\text{and}}\quad
G(Y):=V\otimes_A Y,$$ 
where $U={^{\mu} A^1}(-\bfl)[d]$ and $V=U^{-1}={^1 A^\mu}(\bfl)[-d]$.
All of these are perfect
complexes over $A$ or over $A^{op}$. Note that since $U$ and $V$ are perfect, 
$U\otimes^L_A -$ and $V\otimes^L_A-$ can be computed by $U\otimes_A -$
and $V\otimes_A -$.  Consider the 
following diagram, where $\HA$ (respectively, $\Hk$) is an abbreviation for 
$\Hom_A$ (respectively, $\Hom_k$).
\[
\xymatrix{ M_1 = \HA(X, Y)^* \ar[r]^-{\alpha_1} \ar[d]^-{\psi_1} 
\ar@{}[rd] |{\text{I}} & M_2 = 
\Hk(D(X)\otimes_A G(Y),k) \ar[r]^-{\alpha_2} 
\ar[d]^-{\psi_2} \ar@{}[rd] |{\text{II}} &  \dots \\
M'_1 = \HA(\Sigma X, \Sigma Y)^* \ar[r]_-{\alpha'_1}  & M'_2 
= \Hk(D(\Sigma X)\otimes_A G(\Sigma Y),k) \ar[r]_-{\alpha'_2}  & \dots
}
\]
\[
\xymatrix{ 
\dots M_3 = \HAo(D(X), \Hk(G(Y),k)
) \ar[r]^-{\alpha_3} \ar[d]^-{\psi_3} \ar@{}[rd] |{\text{III}}& M_4 
=  \HAo(D(X), D(G(Y))) \ar[d]^-{\psi_4}   \dots \\
\dots M'_3 = \HAo(D( \Sigma X), \Hk(G(\Sigma Y),k)) 
\ar[r]^-{\alpha'_3} & M_4' = \HAo(D(\Sigma X), D(G( \Sigma Y))) \dots
}
\]
\[
\xymatrix{ \dots \ar[r]^-{\alpha_4} \ar@{}[rd] |{\text{IV}} & M_5 
= \HA(G(Y),X) \ar[r]^-{\alpha_5} \ar[d]^-{\psi_5} 
\ar@{}[rd] |{\text{V}} &  M_6 = \HA(Y,F(X)) \ar[d]^-{\psi_6} \\
\dots \ar[r]^-{\alpha'_4} & M'_5 = \HA(G(\Sigma Y),\Sigma X) 
\ar[r]^-{\alpha'_5}  &  M'_6 = \HA(\Sigma Y,F( \Sigma X)).
}
\]
Here, the upper horizontal maps $\alpha_i$ exist as isomorphisms
in the derived categories, coming from the proof of 
Lemma~\ref{xxlem3.2}(3), and  the lower horizontal maps $\alpha'_i$ 
are those isomorphisms applied to 
the objects $\Sigma X$ and $\Sigma Y$; the vertical maps will be 
defined later in the proof.
By taking $0$-th cohomology and the $0$-th graded piece, we obtain a diagram that involves the Hom
sets in the derived category and isomorphisms denoted formally by  
$H^0(\alpha_i)_0$. 

We now describe the $\alpha_i$ more explicitly.
The map $\alpha_1$ is the inverse of the $k$-linear dual of the map 
$\lambda_1$, where $\lambda_1$ is the composition of the following
isomorphisms 
$$\HA(X,Y)\to H_A(X, F(G(Y)))\to H_A(X, U\otimes_A G(Y))\to 
\HA(X,U)\otimes_A G(Y).$$
The map $\alpha_2$ is just the $\Hom-\otimes$ adjoint. 
The map 
$$\alpha_3:\HAo(D(X), \Hk(G(Y),k))
\to \HAo(D(X), D(G(Y))),$$ 
which exists at the derived level, is 
induced by the isomorphism $\lambda_3: \Hk(G(Y),k)\to D(G(Y))$
in the derived category provided by Lemma~\ref{xxlem3.2}(2) since $G(Y)$ 
is in $\mc{D}_{\epsilon}(A)$.  Note that $\Hk(G(Y),k)\cong \HA(G(Y), A^*)$ 
by adjointness. Next we define a zig-zag of maps that 
become isomorphisms after we take the 
$0$-th cohomology. 
Let $I(U)$ be a fixed $A$-bimodule injective resolution 
of $U$ and let $I(U)^{\leq 0}$ be the truncation.
Then we have quasi-isomorphisms 
$$U\xrightarrow{\lambda_{3,1}} I(U) 
\xleftarrow{\lambda_{3,2}} I(U)^{\leq 0}.$$
We also have 
$$I(U)^{\leq 0}\xleftarrow{\lambda_{3,3}} \Gamma_{\fm}(I(U)^{\leq 0})
\xrightarrow{\lambda_{3,4}}  H^0(\Gamma_{\fm}(I(U)^{\leq 0}))\cong 
A^*,$$
where $\lambda_{3,3}$ is the natural embedding (but not a quasi-isomorphism)
and $\lambda_{3,4}$ is a quasi-isomorphism (following from the
AS Gorenstein property). Since $G(Y)$ is in $\mc{D}_{\epsilon}(A)$, 
$$H_A(G(Y),\lambda_{3,3}): H_A(G(Y), \Gamma_{\fm}(I(U)^{\leq 0}))
\to H_A(G(Y), I(U)^{\leq 0})$$ 
is a quasi-isomorphism. Therefore each $\HA(G(Y),
\lambda_{3,i})$ is a quasi-isomorphism for $i=1,2,3,4$. 
Note that $\lambda_3$ is a composition of the maps $\HA(G(Y),\lambda_{3,i})$
(or their inverses), which becomes a well-defined 
isomorphism in the derived category. Since
$\alpha_3$ is the composition
$$\begin{aligned}
\HAo(D(X),\Hk(G(Y),k))&\to 
\HAo(D(X), \HA(G(Y), A^*))\\
&\xrightarrow{\HAo(D(X), \lambda_3)^{-1}}
\HAo(D(X), \HA(G(Y), U))\\
&\qquad\qquad\qquad\qquad =\HAo(D(X),D(G(Y))),
\end{aligned}
$$
it exists and is an isomorphism at the derived level. 
Finally, $\alpha_4$ is the isomorphism given by the duality $D$, and $\alpha_5$ is another adjoint isomorphism.

For any $X$ in $\D_{\epsilon}(A)$, let $s: X\to \Sigma X$ be the
standard shift map of degree 1, which is the identity on elements
of $X$. This is denoted by $\sigma$ in the notes 
\cite[Section 1.13]{AFH}. For any element $x\in X$, let $s(x)$ be
the corresponding element in $\Sigma X$. 

Now we fix some standard isomorphisms for any perfect 
complexes $X$ and $Y$ (over $A$ or over $A^{op}$), which come 
with important sign conventions.  First, we have an isomorphism
\begin{equation}
\label{E3.3.1}\tag{E3.3.1}
\Sigma: \Hom_A(X, Y) \to \Hom_A(\Sigma X, \Sigma Y), \ \ \ 
f \mapsto (-1)^{|f|} s \circ f \circ s^{-1}.
\end{equation}
Here, we think of $f$ as some homogeneous element of degree 
$j$ in $\Hom_A(X, Y)^j = \prod_i \Hom(X^i, Y^{i+j})$.
The sign, and all signs below, come from the Koszul sign convention.

Similarly, we fix the standard isomorphisms, where $\otimes$ could be
either $\otimes_A$ or $\otimes_k$ and where $\Hom$ could be either
$\Hom_A$, or $\Hom_{A^{op}}$, or $\Hom_k$, 
\begin{gather}
t_1: (\Sigma X \otimes Y) \to \Sigma(X \otimes Y), \ \ \ 
(sx \otimes y) \mapsto  s(x \otimes y), \label{E3.3.2}\tag{E3.3.2}\\
t_2: (X \otimes \Sigma Y) \to \Sigma (X \otimes Y), \ \ \  
(x \otimes sy) \mapsto (-1)^{|x|} s(x \otimes y), \label{E3.3.3}\tag{E3.3.3}\\
(\Sigma^{-1} \otimes \Sigma): ( X \otimes Y) \to 
( \Sigma^{-1} X \otimes \Sigma Y), \ \ \ (x \otimes y) 
\mapsto  (-1)^{|x|}  (s^{-1}x \otimes sy), \label{E3.3.4}\tag{E3.3.4}\\
h_1:  \Hom( \Sigma X, Y) \to \Sigma^{-1} \Hom(X, Y), \ \ \  
f \mapsto (-1)^{|f|} s^{-1}(f \circ s),  
\ \ \text{and}\label{E3.3.5}\tag{E3.3.5}\\
h_2:  \Hom(X, \Sigma Y) \to \Sigma \Hom(X, Y), \ \ \ 
f \mapsto    s (  s^{-1} \circ f),\label{E3.3.6}\tag{E3.3.6}
\end{gather}
all of which are of course natural in each coordinate.   
Since $D = \Hom_A( -, U)$, we get an induced natural isomorphism of functors 
\[
\Omega: \Sigma^{-1} D \to D \Sigma
\]
coming from $h_1^{-1}$.   
Since $G = V \otimes_A  -$ and $F = U \otimes_A -$,  we have 
fixed natural isomorphisms of functors 
\[
\Lambda_G: G \Sigma \to \Sigma G, \ \ \   \Lambda_F: F \Sigma \to \Sigma F
\]
coming from $t_2$.

Now all vertical maps in the diagram above are defined by taking 
the obvious map using a combination of the fixed isomorphisms 
above.  Most importantly, we have $\psi_1 = (\Sigma^*)^{-1}$, 
and 
$\psi_6 = (\Lambda_F)_X^{-1} \circ \Sigma(-)$.  As another 
example, $\psi_2$ is the map 
induced by the isomorphism $DX \otimes^L GY \to D \Sigma X 
\otimes^L G(\Sigma Y)$ given by    $(\Omega_X \otimes \Lambda_G^{-1})   
\circ (\Sigma^{-1} \otimes \Sigma)$.   We leave the similar 
obvious definitions of $\psi_3, \psi_4, \psi_5$ to the reader.

The main technical step of the proof is to check that with the 
definitions of the maps given above, 
then \emph{square I anticommutes {\rm{(}}or $(-1)$-commutes{\rm{)}}, 
while squares II-V all commute.}  
Suppose for the moment that 
this has been verified.   Then taking $H^0$ and dualizing 
everything, we will have shown that 
$\beta_{X, Y} \circ (\Sigma^*)^{-1} \circ \alpha_{X, Y} 
= - (\alpha_{\Sigma X, \Sigma Y} \circ \Sigma)$, 
where $\beta_{X,Y}$ is the dual of the map 
$\Hom_{\mc{D}_{\epsilon}(A)}(\Sigma Y, \Sigma F X) \to 
\Hom_{\mc{D}_{\epsilon}(A)}(\Sigma Y, F \Sigma X)$
given by $f \mapsto (\Lambda_F)_X \circ f$.  In other words, 
this shows that $\eta_{-1}(\Sigma) = \Lambda_F$. 
However, because $U$ is a complex concentrated entirely in 
degree $d$, the sign in the isomorphism $t_2$ 
from which $\Lambda_F$ is derived implies that $\Lambda_F$ 
is equal to $(-1)^d \1_{\Sigma F}$, where 
$\1_{\Sigma F}: \Sigma F \to F \Sigma$ is the trivial 
commutation.  This will verify the result for 
$\eta_{-1}(\Sigma)$.  

The verification that square I anticommutes, while 
each square II-V commutes, is tedious, and we do not give 
all of the details, but we briefly sketch some parts.  
Square II is easily seen to commute using 
the functoriality of the adjoint isomorphism.    The commutation 
of square IV is mostly a matter of using 
that the duality $D$ is a functor.  Square V commutes 
again since $F$ and $G$ are adjoint (in fact inverse equivalences), 
and the natural isomorphisms $\Lambda_F, \Lambda_G$ used to 
commute these with $\Sigma$ are defined in a way compatible 
with this adjointness.  

The more difficult squares are I and III.  
If necessary, each square can be analyzed by carefully 
writing out the definitions of the maps and applying them 
to a test element, and this is what we resorted to in order 
to verify that square I anticommutes.  
The commutation of square III depends on the maps $\lambda_{3,i}$
introduced earlier.

For those readers who would like to see more details, we now give a 
full proof of the anticommutativity of square I.  We leave the more detailed verification of
 the commutativity of square III (and the other squares) to the reader.
Note that the $k$-linear dual of square I is the composition
of the following three squares
\[
\xymatrix{  \HA(X, Y)\ar[r]^-{\lambda_{1,1}} \ar[d]^-{\Sigma} 
\ar@{}[rd] |{\text{I$_1$}}
& \HA(X, FG(Y)) \ar[d]^-{w_2}  \\
\HA(\Sigma X, \Sigma Y) \ar[r]_-{\lambda'_{1,1}}  & 
\HA(\Sigma X, FG(\Sigma Y)) ,
}
\]
\[
\xymatrix{  \HA(X, FG(Y))\ar[r]^-{\lambda_{1,2}} \ar[d]^-{w_2} 
\ar@{}[rd] |{\text{I$_2$}}
& \HA(X, U\otimes_A G(Y)) \ar[d]^-{w_3}  \\
\HA(\Sigma X, FG(\Sigma Y)) \ar[r]_-{\lambda'_{1,2}}  & 
\HA(\Sigma X, U\otimes_A G(\Sigma Y)), 
}
\]
and
\[
\xymatrix{  \HA(X, U\otimes_A G(Y))\ar[r]^-{\lambda_{1,3}} \ar[d]^-{w_3} 
\ar@{}[rd] |{\text{I$_3$}}
& \HA(X, U)\otimes_A G(Y) \ar[d]^-{w_4}  \\
\HA(\Sigma X, U\otimes_A G(\Sigma Y)) \ar[r]_-{\lambda'_{1,3}}  & 
\HA(\Sigma X, U)\otimes_A G(\Sigma Y) .
}
\]
The map $\Sigma$ is given in \eqref{E3.3.1}.  The map 
$\lambda_{1,1}$ sends $f\in \HA(X,Y)$ to $\eta_Y\circ f$, where $\eta_Y:
Y\to FG(Y)$ is the natural isomorphism sending $y\to u\otimes v\otimes y$, 
where $u$ and $v$ are the (nonzero) generators of $U$ and $V$.  The map 
$w_2$ is the the composition of $\Sigma$ with $\HA(\Sigma X, \delta_Y)$,
where $\delta_Y: \Sigma FG(Y)\to FG(\Sigma Y)$ sends 
$s(u\otimes v\otimes y)\to u\otimes v\otimes s(y)$. Note that  
$\delta_Y$ is a composition of two different $t_2$'s. For any
$f\in \HA(X,Y)$ and $sx\in \Sigma X$, we have
$$\begin{aligned}
\;[(w_2\circ \lambda_{11})(f)](sx)&=w_2 (\eta_Y\circ f)(sx)\\
&=\HA(\Sigma X, \delta_Y) [(-1)^{|f|} s\circ (\eta_Y\circ f) \circ s^{-1}
(sx)]\\
&=\HA(\Sigma X, \delta_Y) [(-1)^{|f|} s\circ (\eta_Y\circ f)(x)]\\
&=\HA(\Sigma X, \delta_Y) [(-1)^{|f|} s\circ (u\otimes v\otimes f(x))]\\
&=(-1)^{|f|} u\otimes v\otimes sf(x), \qquad\qquad \qquad\qquad 
{\text{and}}\\
\; [(\lambda'_{11}\circ \Sigma)(f)](sx)&=
[\lambda'_{11}((-1)^{|f|} s\circ f \circ s^{-1}](sx)\\
&=(-1)^{|f|} \eta_{\Sigma Y}(s f(x))\\
&=(-1)^{|f|} u\otimes v\otimes sf(x). 
\end{aligned}
$$
So $w_2\circ \lambda_{11}=\lambda'_{11}\circ \Sigma$ and square I$_1$
is commutative. By definition, $F(GY)=U\otimes_A G(Y)$ (so 
$\lambda_{1,2}$ is the identity) and the map 
$w_3$ (like $w_2$) is the composition of $\Sigma$ with 
$\HA(\Sigma X, \eta_Y)$, where $\eta_Y$ is viewed as a natural 
isomorphism from $\Sigma (U\otimes G(Y))\to U\otimes G(\Sigma)$.
It is clear that square I$_2$ is commutative. Finally we show that
square I$_3$ is anticommutative. The map $\lambda_{1,3}^{-1}: 
\HA(X,U)\otimes_A G(Y)\to \HA(X,U\otimes_A G(Y))$ is defined
as follows: for any homogeneous element $f\otimes (v\otimes y)\in 
\HA(X,U)\otimes_A G(Y)$, $\lambda_{1,3}^{-1}(f\otimes (v\otimes y))$
in $\HA(X,U\otimes_A G(Y))$ is denoted by $\{f,v,y\}$ and for 
any homogeneous element $x\in X$,
$$\{f,v,y\} (x)=(-1)^{|x|(d+|y|)} f(x) \otimes (v\otimes y).$$
(Note that $d=|v|$.) 
So for $g\otimes (v\otimes s y)\in \HA(\Sigma X, U)\otimes_A G(\Sigma Y)$ 
and $sx\in \Sigma X$, we have
$$(\lambda'_{1,3})^{-1}(g \otimes (v\otimes sy)) (sx)=
(-1)^{(|x|-1)(d+|y|-1)} g(sx)\otimes (v\otimes sy).$$
 We have already seen the definition of $w_3$.
The map $w_4$ is the composition of $\Sigma^{-1}\otimes \Sigma$
with $h_1^{-1}\otimes t_2^{-1}$, so it sends 
$$f\otimes v\otimes y\mapsto (-1)^{d} (f\circ s^{-1})
\otimes v\otimes sy.$$
Then
$$\begin{aligned}
\; [ (\lambda_{1,3}')^{-1}\circ w_4(f\otimes v\otimes y)](sx)
&= [(\lambda_{1,3}')^{-1}((-1)^{d} (f\circ s^{-1})\otimes v\otimes sy)](sx)\\
&=(-1)^{d+(|x|-1)(d+|y|-1)} (f\circ s^{-1})(sx)\otimes v\otimes sy\\
&=(-1)^{\omega_1} f(x)\otimes v\otimes sy
\end{aligned}
$$
where
$$\omega_1\equiv d+(|x|-1)(d+|y|-1)\equiv |x|(|d|+|y|)+|x|+|y|+1
\quad \mod 2.$$
On the other hand,
$$\begin{aligned}
\; [w_3\circ (\lambda_{1,3})^{-1}(f\otimes v\otimes y)] (sx)
&= [w_3 (\{f,v,y\})](sx)\\
&= \HA(\Sigma X,\delta_Y)[(-1)^{|f|+d+|y|} s \circ \{f,v,y\} 
\circ s^{-1}](sx)\\
&= (-1)^{|f|+d+|y|}\HA(\Sigma X,\delta_Y)(s \circ \{f,v,y\}(x))\\
&= (-1)^{|f|+d+|y|+|x|(d+|y|)}
\HA(\Sigma X,\delta_Y)(s(f(x)\otimes v\otimes y))\\
&= (-1)^{\omega_2}f(x)\otimes v\otimes sy
\end{aligned}
$$
where $\omega_2=|f|+d+|y|+|x|(d+|y|)$. Since 
$|f|=|f(x)|-|x|=|u|-|x|=-d-|x|$, we have
$$\omega_2=-|x|+|y|+|x|(d+|y|)=\omega_1-1 \mod 2.$$
Therefore $(\lambda'_{1,3})^{-1}\circ w_4=-w_3\circ (\lambda_{1,3})^{-1}$
and square I$_3$ is anticommutative. Combining squares I$_1$, I$_2$
and I$_3$ and then taking the $k$-linear dual, one sees that square I 
is anticommutative. Therefore we verify square I.

A similar proof, but much easier, will show that 
$\eta_1(T) = \1_{ TF}$.  In the corresponding diagram, there are no 
signs to worry about.  Each square will easily be seen 
to commute.  Thus we omit this simpler part of the proof.   
Note that it is immediate from the definitions that $T$ 
and $\Phi$ are exact 
functors of $\mc{D}_{\epsilon}(A)$ using the trivial 
commutations.  Thus Hypothesis~\ref{xxhyp2.5} is 
verified, with values $s = (-1)^d$ and $t = 1$.  So  $(\mc{D}_{\epsilon}(A), \Sigma, T)$ 
is a $\Phi$-skew Calabi-Yau triangulated category by definition.
\end{proof}

\begin{remark}
\label{xxrem3.4}
The values of $s$ and $t$ we calculated in the preceding 
result are of course the natural and expected ones.  
As we saw in Section~\ref{xxsec2}, $-\1_{\Sigma^2}: 
\Sigma^2 \to \Sigma^2$ makes $\Sigma$ into an exact 
functor of $(\mc{D}(A), \Sigma)$, whereas the natural 
commutations with $\Sigma$ make $T$ and $\Phi$ into 
exact functors.  Since $F = \Sigma^d \circ T^{-\bfl} \circ \Phi$, 
then $(-1)^d \1_{F \Sigma}: F \Sigma \to \Sigma F$ makes $F$ exact.

However, although Van den Bergh proves in \cite[Theorem A.4.4]{Bo} 
that the natural transformation $\eta_{-1}(\Sigma)$ always makes $F$ into an exact 
functor, there is no apparent reason why there should 
be a unique natural transformation that does so in 
general.  Thus $(-1)^d \1_{F \Sigma} = \eta_{-1}(\Sigma)$ is not obviously forced.

In addition, as we discussed in Section~\ref{xxsec2.3}, 
the exact form of $\eta_{-1}(\Sigma)$ may depend on the choice of 
the isomorphisms $\alpha_{X, Y}$ in the Serre duality, if the category has a large center.  
The proof of  Proposition~\ref{xxpro3.3} suggests that 
the choices of these isomorphisms $\alpha_{X, Y}$ exhibited in 
Lemma~\ref{xxlem3.2}(3) are particularly natural ones, since they lead to 
the expected $\eta_{-1}(\Sigma)$.
\end{remark}

In the remaining results in this section, we continue to maintain the notation introduced before 
Proposition~\ref{xxpro3.3}.
We now prove the main result of this section.
\begin{theorem}
\label{xxthm3.5}
Let $A$ be a generalized AS Gorenstein algebra.
Then $(\mc{D}^{pl}_{\epsilon}(A), \Sigma, T)$ is a skew 
Calabi-Yau triangulated category. 
\end{theorem}

\begin{proof} To avoid triviality, we assume that 
$\mc{D}^{pl}_{\epsilon}(A)$ is not zero. 
By Proposition \ref{xxpro3.3}, $(\mc{D}_{\epsilon}(A), \Sigma, T)$ is a $\Phi$-skew 
Calabi-Yau triangulated category. The assertion follows
from Lemma \ref{xxlem2.10}.
\end{proof}

Next, we see that for the skew Calabi-Yau categories which arise in this section, the 
dimension, AS index, and Nakayama functor are uniquely determined.
\begin{lemma}
\label{xxlem3.6}
Let $A$ be an $\mb{N}$-graded generalized AS Gorenstein algebra, of 
dimension $d$ and of AS index $\bfl$, with Nakayama 
automorphism $\mu = \mu_A$.   
\begin{enumerate}
\item
Let $G=\Sigma^a \circ T^b \circ \Phi$ for some $a, b \in \mb{Z}$.  If $X$ is a nonzero object in
$\D_{\epsilon}(A)$ and $G(X)\cong X$, then $a=b=0$.
\item
If we require the existence of a nonzero $\Phi$-plain 
object in $\D_{\epsilon}(A)$, 
then the decomposition \eqref{E2.9.1} is unique, namely, 
$(d, \bfl, \Phi)$ is uniquely determined by \eqref{E2.9.1}.
\end{enumerate}
\end{lemma}
\begin{proof} (1) Note that $H^*(\Phi(X))=H^*(X)$ as graded
$k$-spaces. Since $G(X)\cong X$, we have $\Sigma^a T^b H^*(X)
\cong H^*(X)$. Since $X$ is a bounded complex of graded 
finitely generated left $A$-modules, $H^*(X)$ is left and right 
bounded (using the complex degree) and bounded below (using
the internal grading). If $a\neq 0$ or $b\neq 0$ and if $G(X)
\cong X$, we have that $H^*(X)=0$. Equivalently, $X=0$ in
$\D_{\epsilon}(A)$. The assertion follows.

(2) Suppose that there is another decomposition $F=\Sigma^{d'}\circ
T^{\bfl'}\circ \Phi'$ such that there is a nonzero $\Phi'$-plain
object in $\D_{\epsilon}(A)$. Then $\Phi'=\Sigma^{d-d'}\circ
T^{\bfl-\bfl'}\circ \Phi$ and there is an $X\neq 0$ such that
$\Phi'(X)\cong X$. By part (a), $d'-d=0=\bfl'-\bfl$. Thus 
$\Phi'=\Phi$. 
\end{proof}

Suppose that $A$ is a connected $\mb{N}$-graded algebra. It was shown 
in~\cite[Corollary~D]{LPWZ} that $A$ is AS regular if and only if 
the Ext-algebra $E(k)$ is Frobenius; this generalized the 
well-known result of Smith in the Koszul case~\cite{Smith}. 
To close this section, we will show that if $A$ is generalized AS regular, then the Ext-algebra
$E(A_0)$ is Frobenius, where $A_0$ denotes the left (and right)
graded $A$-module $A/A_{\geq 1}$. 

\begin{proposition}
\label{xxpro3.7}
Suppose that $A$ is an $\mb{N}$-graded generalized AS regular algebra of dimension $d$, 
with Nakayama automorphism $\mu = \mu_A$.
\begin{enumerate}
\item
$A_0 = A/A_{\geq 1}$, considered as a complex concentrated in degree $0$, is a nonzero $\Phi$-plain 
object in $\D_{\epsilon}(A)$.
\item
If $A_0$ is semisimple, then
$\D_\epsilon(A)= \D^{pl}_\epsilon(A)$ is a skew Calabi-Yau triangulated 
category of dimension~$d$ with Nakayama functor 
$\Phi = {}^\mu A^1 \otimes -$.
\item
$E(A_0) := \bigoplus_{i,j} \Ext^i_A(A_0, A_0(j))$ is a (finite 
dimensional) $\Z$-graded Frobenius algebra.
\end{enumerate}
\end{proposition}
\begin{proof}
(1) Since $\mu$ is a graded algebra automorphism, 
$\Phi(A_0)\cong A_0$.  Since $A$ is generalized AS regular, in particular noetherian 
of finite global dimension, all complexes in ${\mathcal D}^b_{f}(A)$ are perfect.  
Thus $D_{\epsilon}(A)$ is the same as $\mc{E}(A)$, the subcategory of ${\mathcal D}^b_{f}(A)$ 
consisting of complexes with finite-dimensional cohomologies.  In particular, $A_0 \in D_{\epsilon}(A)$.

(2) By Theorem~\ref{xxthm3.5}, it suffices to show that $\Phi$ is plain.  
By part (1), $A_0\in \Xi(\Phi)$.  Then the thick subcategory $\D_{\epsilon}^{pl}$ generated by $\Xi(\Phi)$ contains 
all direct summands of $A_0$, and since $A_0$ is semisimple, $\D_{\epsilon}^{pl}$ contains all simple $A_0$-modules $M$.
We saw in the proof of  Lemma~\ref{xxlem2.10} that $\D_{\epsilon}^{pl}$ is closed under $T$, so $\D_{\epsilon}^{pl}$ 
contains $T^j(M)$ for all such $M$.   Now it is easy to see that the 
triangulated category generated by all of these complexes is $\mc{E}(A)$.  So 
$\D_{\epsilon}^{pl} = \mc{E}(A) = D_{\epsilon}(A)$.
Therefore $\Phi$ is plain. 

(3) By part (2), $\D_\epsilon(A)$ is a skew Calabi-Yau 
category with Nakayama functor $\Phi = {}^\mu A^1 \otimes -$. 
  It follows from the noetherian condition that $E(A_0)$ is finite 
dimensional.  By standard facts about the derived category, we can identify $E(A_0)$ with $\bigoplus_{i,j} \Hom_{\D_{\epsilon}(A)}(A_0, \Sigma^i T^j(A_0))$.
Since $\Phi(A_0) \cong A_0$ by part (1), $E(A_0)$ is Frobenius by Theorem \ref{xxthm2.7}. 
\end{proof}

\section{Applications of the formula for the Nakayama 
automorphism of $E$}
\label{xxsec4}

In the previous section, we showed that given any generalized AS 
Gorenstein algebra $A$, we can associate a $\Phi$-skew Calabi-Yau 
triangulated category $\mc{B} = \D_{\epsilon}(A)$, the full 
subcategory of the bounded derived category of finitely generated 
graded $A$-modules consisting of perfect complexes with 
finite-dimensional cohomology.  In this section, we show how 
one can obtain interesting information about $A$ by using Theorem~\ref{xxthm2.7} to study the 
Ext-algebras associated to objects $X$ in this category.   

First, note that in the case that $A$ does not have finite global 
dimension, it is conceivable that $\mc{B}$ may be the zero category, 
in which case it carries no interesting information.  However, in 
all of the examples we know, $\mc{B}$ is nonzero and we conjecture that 
this is always the case.  In Section~\ref{xxsec5} we will show 
that $\mc{B}$ is nonzero in many important common cases.  

Theorem~\ref{xxthm2.7} applies only to a $\Phi$-plain object $X$.
Recall that $\Phi$ is induced by the twist $M \mapsto {}^{\mu} M$ on left modules in this case, where $\mu= \mu_A$ is the Nakayama automorphism of $A$.  Thus we need $\mc{B}^{pl}$ to be nonzero.  In the remainder of the section, we explore two different cases where this holds:  when $\mu_A$ is of a special form, or when  $A$ has finite global dimension.

\subsection{Application one}
\label{xxsec4.1}
Let $A$ be noetherian connected graded AS Gorenstein. Any Hopf 
algebra $H$ acting on $A$ such that $A$ is a left $H$-module 
algebra determines a map $\hdet: H \to k$, the 
\emph{homological determinant} \cite[Definition 3.3]{KKZ}, 
which gives important information about the action.  In this paper we only study the case 
where $H= kG$ is a group algebra for a group $G$ of graded 
automorphisms of $A$.
Recall from the previous section that the rigid dualizing complex of 
$A$ is of the form
\[
R = R^d\Gamma_{\mf{m}_A}(A)^*[d] \cong {}^{\mu} A^1(-{\bfl})[d],
\]
where $\mu = \mu_A$.  
Now since $A$ has a left $kG$-module structure, $R^d\Gamma(A)^*$ also 
obtains a left $kG$-module structure; see \cite[Section 3]{KKZ} or \cite[Remark 3.8]{RRZ} for more details.  
Thus via the isomorphism above, ${}^{\mu} A^1(-{\bfl})[d]$ has a left 
$kG$-action.  Each element of $kG$ acts on the $1$-dimensional degree ${\bfl}$ 
piece by a scalar, and this assignment defines the homological 
determinant $\hdet = \hdet(A): kG \to k$.   We do not give the precise definition since it 
is not needed in the proofs below.  We will, however, rely on some theorems about $\hdet$ 
which are established in \cite{RRZ}.

Recall that $|x|$ indicates the degree of a homogeneous element $x$.
For any $\mb{Z}$-graded algebra $R$ and $c \in k^{\times}$, we may define 
a graded automorphism $\xi_c$ of $R$ where $\xi_c(x) = c^{|x|} x$ for 
homogeneous elements $x$.  Similarly, 
if $R$ is $\mb{Z}^2$-graded 
then $\xi_{c,d}$ is the $\mb{Z}^2$-graded automorphism with 
$\xi_{c,d}(x) = c^id^j x$ for homogeneous elements $x$ of degree $|x| = (i,j)$.
This definition extends in an obvious way to multi-gradings.

Now suppose that $A$ is connected graded noetherian AS Gorenstein as above, and that $\mu = \mu_A$ happens to have the form $\xi_c$ for some $c \in k^\times$.  While this seems quite restrictive, we will see in the proof of Theorem~\ref{xxthm5.3} that any 
noetherian AS Gorenstein algebra is closely related to another one with a Nakayama automorphism of this simple form.
For any $\mb{Z}$-graded $A$-module $M$, 
the map $\rho_M: M \to \Phi(M) = {} ^{\mu} M$ given by 
$m \mapsto c^{|m|} m$ for homogeneous elements $m$ is an isomorphism 
of graded left $A$-modules.   This extends in an obvious way 
to complexes and thus to objects of the derived category $\mc{D}(A)$ and its subcategory $\mc{B} = \D_{\epsilon}(A)$.  
Thus for any object $X \in \mc{B}$ there is an isomorphism $\rho_X: X \to \Phi(X)$, and so we 
have a natural isomorphism of functors $\rho: \1 \to \Phi$.   In particular, we can apply Theorem~\ref{xxthm2.7} 
to obtain the following result about the bigraded Ext-algebra of $X$.

\begin{theorem}
\label{xxthm4.1}
Let $A$ be noetherian connected graded AS Gorenstein, with 
$\mu_A = \xi_c$ for some $c \in k^{\times}$.  Assume that 
${\mathcal B}:=\D_{\epsilon}(A) \neq 0$.  Let $d$ be the injective dimension of $A$, and let $\bfl$ be 
the AS index of $A$.
\begin{enumerate}
\item
For any nonzero $X \in {\mathcal B}$, the Ext-algebra 
$E = \bigoplus_{i,j} \Hom_{\mathcal B}(X, \Sigma^i T^j X)$
is Frobenius with bigraded Nakayama automorphism 
$\mu_E = \xi_{(-1)^{d+1}, c^{-1}}$.
\item
$c^{{\bfl}} = 1$, and $\hdet \mu_A = c^{{\bfl}} = 1$.
\end{enumerate}
\end{theorem}

\begin{proof}
(1) As we have seen, there is an isomorphism $\rho_X: X \to \Phi(X)$.
We also know by Proposition~\ref{xxpro3.3} that $(\mc{B}, \Sigma, T)$ satisfies Hypothesis~\ref{xxhyp2.5} 
with $s = (-1)^d$ and $t = 1$, and so 
Theorem~\ref{xxthm2.7} applies, where the Serre functor $F$ of $\mc{B}$ has the form $\Sigma^d T^{-\ell} \Phi$.

It follows easily from the definitions that
for any object $X$,
\begin{equation}
\label{E4.1.1}\tag{E4.1.1}
\Sigma (\rho_X) =\rho_{\Sigma(X)} \quad {\text{and}}\quad
T (\rho_X) =c \;\rho_{T(X)}.
\end{equation}
Now let $X$ be any nonzero object in $\mc{B}$ and let $\phi = \rho_X:  X\to \Phi(X)$. 
Since $X$ is perfect, by replacing $X$ with a quasi-isomorphic complex of projectives, we can think of
$\phi$ as well as elements in $\Hom_{\mc{B}}(X, \Sigma^i T^j X)$ as actual maps of complexes.
Let 
$g \in E_{i,j} = \bigoplus_{i, j \in \mb{Z}} \Hom_{\mc{B}}(X, \Sigma^i T^j X)$
be a homogeneous element of degree $(i,j)$. Then
\[
\Sigma^i T^j(\phi)^{-1} \circ \Phi(g) \circ \phi
= \Sigma^i T^j(\rho_X)^{-1} \circ \Phi(g) \circ \rho_X
= c^{-j} (\rho_{\Sigma^i T^j X})^{-1} \circ \Phi(g) \circ \rho_X,
\]
using \eqref{E4.1.1}.  
Let $X^n$ be the $n$th term in the complex $X$ and let $x \in (X^n)_m$ be a homogeneous element 
of degree $m$ in this module.  Then we calculate 
\[
c^{-j} (\rho_{\Sigma^i T^j X})^{-1} \circ \Phi(g) \circ \rho_X(x) = c^{-j}   c^{-m} g(c^m x) = c^{-j} g(x).
\]
Since this is independent of $n$ and $m$, we see that $\Sigma^i T^j(\phi)^{-1} \circ \Phi(g) \circ \phi = c^{-j} g$, 
as morphisms of complexes.

By Theorem~\ref{xxthm2.7}, the action of the Nakayama automorphism $\mu_E$ on the $(i,j)$-graded 
piece is multiplication by the scalar $(-1)^{(d+1)i} c^{-j}$.  In other words, thinking of 
$E$ as $\mb{Z}^2$-graded, we have 
$$\mu_E = \xi_{(-1)^{d+1}, c^{-1}}.$$

(2) Since $E$ is Frobenius with Nakayama automorphism $\mu = \mu_E$, 
there is a nondegenerate bilinear form $\langle -, - \rangle$ on $E$,
such that 
$$\langle a, b \rangle = \langle \mu(b), a \rangle$$ 
for $a, b \in E$. Let $\mf{e} = \langle 1, - \rangle$.  Then under the 
natural action of $G = \Autz(E)$ on $E^* = \Hom_k(E, k)$, we have 
\[
\mu(\mf{e})(b) = \mf{e}(\mu^{-1}(b)) = 
\langle 1, \mu^{-1}(b) \rangle = \langle b, 1 \rangle = \langle 1, b \rangle =  \mf{e}(b) 
\]
and thus $\mu(\mf{e}) = \mf{e}$.
Note that $E^*$ is $\mb{Z}^2$ graded with $(E^*)_{i,j} = \Hom_k(E_{-i,-j}, k)$.
Since $\mu_E = \xi_{(-1)^{d+1}, c^{-1}}$, it is 
easy to see that $\mu$ also must act on $E^*$ via the same rule 
$x \mapsto (-1)^{i(d+1)} c^{-j} x$ for $|x| = (i,j)$.  Since 
$\mf{e}$ has degree $(-d, {\bfl})$ we also have 
$\mu(\mf{e}) = (-1)^{-(d+1)d} c^{-{\bfl}} \mf{e} = c^{-{\bfl}} \mf{e}$.  Thus 
$c^{{\bfl}} = 1$.

Finally, since $\mu_A = \xi_c$ and $A$ has AS index ${\bfl}$, it also 
follows that $\hdet_A(\mu_A) = c^{{\bfl}}$ \cite[Lemma 5.3]{RRZ}. 
Thus $\hdet(\mu_A) = 1$.
\end{proof}

In Section~\ref{xxsec5}, we will extend the result $\hdet \mu_A = 1$ to 
all AS Gorenstein algebras with $\D_{\epsilon}(A) \neq 0$ 
[Theorem \ref{xxthm5.3}].

\subsection{Application two}
\label{xxsec4.2}

Throughout this application, we assume that $A$ is an $\mb{N}$-graded 
generalized AS regular algebra, as defined in Definition~\ref{xxdef3.1}.
We also assume that $A_0$ is a semisimple $k$-algebra.  
Let $X$ be the trivial module $A_0 := A/A_{\geq 1}$, which 
is viewed as a complex concentrated in degree $0$.   Let $\mc{B} = \D_{\epsilon}(A)$ be defined as above. 
As already proved in Proposition~\ref{xxpro3.7}, the bigraded Ext-algebra
\[
E = \bigoplus_{i,j} \Hom_{\mc{B}}(X, \Sigma^i T^j X) \cong 
\bigoplus_i \Ext^i_A(A_0, A_0)
\]
is a graded Frobenius algebra.    Theorem~\ref{xxthm2.7} also gives a formula for the Nakayama 
automorphism $\mu_E$ of $E$.  In this section we study the formula for $\mu_E$ more closely 
and show that it recovers and generalizes some existing results in the literature on the Nakayama automorphisms 
of Ext-algebras of regular algebras.

We first set up some notation.  
Let 
\[
P = 0 \to P^{(-d)} \to  \dots \to P^{(-2)} \overset{\partial^{(-2)}}{\to} P^{(-1)} 
\overset{\partial^{(-1)}}{\to} P^{(0)} \to 0
\]
 be a minimal graded projective 
resolution of $X=A_0$, where  each term is of the form 
$P^{(-i)} = A \otimes_{A_0} V^{(-i)}$ for some finite-dimensional
graded left $A_0$-module $V^{(-i)}$.      Note that throughout this section, we will place parentheses 
around the superscripts indicating cohomological degrees, in order to avoid notational confusion with exponents.
Clearly, $V^{(0)} = A_0$.   Explicitly, one may construct the $P^{(-i)}$ and $\partial^{(-i)}$ inductively 
by letting $K^{(-i+1)} = \ker \partial^{(-i+1)}$ (or $K^{(0)} = A_{\geq 1}$ when $i =1$), taking 
$V^{(-i)} = K^{(-i+1)}/ A_{\geq 1} K^{(-i+1)}$, and letting $\partial^{(-i)}: P^{(-i)} = A \otimes_{A_0} V^{(-i)} \to K^{(-i+1)}$ be defined by choosing some $A_0$-module map $j: V^{(-i)} \to K^{(-i+1)}$ such that $\pi \circ j = 1$, 
where $\pi: K^{(-i+1)} \to V^{(-i)}$ is the natural surjection.

Let $\sigma \in \Autz(A)$. This restricts to an algebra isomorphism
$\sigma \colon A_0 \to A_0$, and one has an isomorphism of left
$A$-modules $A_0 \cong {}^\sigma A_0$ (with underlying set map
given by $a \mapsto \sigma(a)$). 
This may be lifted to a graded isomorphism of complexes $\phi: P \to {}^{\sigma} P$.   
Identifying the underlying vector space of $^{\sigma} P^{(-i)}$ with 
$P^{(-i)}$ for each $i$, 
we may also think of $\phi: P \to P$ as a \emph{$\sigma$-linear} 
map of complexes; that  is, each map $\phi^{(-i)}: P^{(-i)} \to P^{(-i)}$ 
satisfies $\phi^{(-i)}( ax ) = \sigma(a) x$ for $x \in P^{(-i)}$, $a \in A$.

It is easy to calculate $\phi^{(0)}$ and $\phi^{(1)}$ explicitly.
Identifying $P^{(0)} = A$, clearly we may take $\phi^{(0)} = \sigma: A \to A$. 
By the description of the formation of the minimal projective resolution given above, we may take $V^{(-1)} = A_{\geq 1}/A_{\geq 1} A_{\geq 1}$ as a left $A_0$-module.  
Then $\sigma$ restricts to $A_{\geq 1}$ and factors through to give a map $\overline{\sigma}: V^{(-1)} \to V^{(-1)}$, 
which induces the map $\phi^{(1)} : P^{(-1)} \to P^{(-1)}$ by applying $A \otimes_{A_0} - $.
For example, if $A$ is generated in degrees $0$ and $1$, then we may identify $V^{(-1)}$ with 
$A_{\geq 1}/A_{\geq 2} = A_1$ and we simply have $\overline{\sigma} = \sigma \vert_{A_1}$.  
If one has an explicit graded presentation of $A$, it is also straightforward to calculate $\phi^{(2)}$ 
explicitly in terms of the action of $\sigma$ on the relations, but  it could be computationally difficult to find the entire $\sigma$-linear map of complexes $\phi$ explicitly when the global dimension of $A$ is large.

Next, we recall that each $\sigma\in \Autz(A)$ induces a (bi)-graded automorphism
$f_\sigma$ of the Ext-algebra $E = \bigoplus_{i \geq 0} \Ext^i_A(A_0, A_0)$, 
so that the group $\Autz(A)^{op}$ acts on $E$ naturally.  
(In the case when $A$ is connected, this graded automorphism $f_\sigma$ 
of the Ext-algebra $E = \bigoplus_{i \geq 0} \Ext^i_A(k, k)$ was defined 
in~\cite[Lemma 4.2]{JZ}; see also \cite[Remark 5.11(a)]{KKZ}.) The map 
$f_{\sigma}$ is described in the following 
simple way.  Fix a $\sigma$-linear isomorphism of complexes 
$\phi: P \to P$ as above.  Since $P$ is a minimal graded projective resolution, 
$\partial^{(i)}(P^{(i)}) \subseteq A_{\geq 1} P^{(i-1)}$ for each $i$.
Thus the differentials in the complex $\Hom_A(P, A_0)$ are $0$, 
and hence we may identify  $\Ext^i_A(A_0, A_0)$ with $\Hom_A(P^{(-i)}, A_0)$ 
for each $i$.  Then $f_\sigma: \Hom(P^{(-i)}, A_0) \to \Hom(P^{(-i)}, A_0)$ 
is given by  $g \mapsto \sigma^{-1} \circ g \circ \phi^{(-i)}$. (If $g$ is 
$A$-linear, then $g \circ \phi^{(-i)}$ is $\sigma$-linear and 
$\sigma^{-1} \circ g \circ \phi^{(-i)}$ is again $A$-linear.)

We use the notation $M^{\vee} = \Hom_{A_0}( - , A_0)$ 
for the dual of a graded left $A_0$-module $M$.  Of course, this is the usual vector space dual $M^*$ 
when $A$ is connected.
Note that we also have canonical isomorphisms 
$$
\Hom_A(P^{(-i)}, A_0) \cong \Hom_A(P^{(-i)}/A_{\geq 1} P^{(-i)}, A_0) 
\cong \Hom_{A_0}( V^{(-i)}, A_0) = (V^{(-i)})^{\vee}.
$$
In particular, this identifies $\Ext^1_A(A_0, A_0)$ with $(V^{(-1)})^{\vee}$.
In the special case that $A$ is generated in degrees $0$ and $1$, we saw above that we can also identify 
$V^{(-1)}$ with $A_1$, and hence $\Ext^1_A(A_0, A_0)$ is identified with $(A_1)^{\vee}$.
Now combining the observations above with Theorem~\ref{xxthm2.7}, 
we obtain the main result of this section.

\begin{theorem} 
\label{xxthm4.2}
Let $A$ be 
generalized AS regular of global dimension $d$, where $A_0$ is semisimple.  Let $\Autz(A)^{op}$ act on the 
Ext-algebra $E(A_0) = \bigoplus_{i \geq 0} \Ext^i_A(A_0, A_0)$ 
as defined above.
\begin{enumerate}
\item
$E$ is Frobenius and 
$\mu_E = \xi_{(-1)^{d+1}, 1} \circ f_{\mu_A}$.  

\item
Suppose that $A$ is generated as an algebra in degrees $0$ and $1$.  
Under the natural identification of $E^1 = \Ext^1_A(A_0, A_0)$  with $(A_1)^{\vee} = \Hom_{A_0}(A_1, A_0)$, the map 
$\mu_E \vert_{E^1}$ is identified with $g \mapsto (-1)^{d+1} (\mu_A \vert_{A_0})^{-1} \circ  g \circ \mu_A \vert_{A_1}$.

\item If $A$ is connected and generated in degree $1$, then 
$$\mu_E \vert_{E^1} = (-1)^{d+1} (\mu_A \vert_{A_1})^*.$$
Moreover, $\mu_A = 1$ if and only if $E$ is graded-symmetric. 
\end{enumerate}
\end{theorem}

\begin{proof}
(1)  We need to relate the action of $\mu_E$ on the Ext-algebra 
described in Theorem~\ref{xxthm2.7} in terms of the derived 
category to the action of $\mu_A$ defined above.   Keep all of the 
above notation; in particular, fix the minimal projective resolution $P$ 
of $A_0$.   As mentioned earlier several times, 
by standard facts about the derived category, $\Ext^i_A(A_0, A_0)_j$ may be identified with 
$\Hom_{\mc{D}_{\epsilon}(A)}(P,  \Sigma^i T^j P)$.   Recall that the 
explicit identification is made as follows.   As we already saw above, an element of 
$\Ext^i_A(A_0, A_0)_j$ corresponds to an element in 
$\Hom_{A}(P^{(-i)}, T^j(A_0))$.  Any element $G(i) \in \Hom_{A}(P^{(-i)},T^j( A_0))$ 
extends trivially to a map of complexes 
$$G :P\to \Sigma^i T^j (A_0).$$ 
Using the quasi-isomorphism 
$\Sigma^i T^j P \to \Sigma^i T^j A_0$,  
$G$  lifts to a map of complexes $g: P \to \Sigma^i T^j P$.  The $g$ so 
obtained is unique up to homotopy. 
Conversely, since $P$ is perfect, an element of $\Hom_{\mc{D}_{\epsilon}(A)}(P, \Sigma^iT^jP)$ 
is given by an actual map of complexes $P \to \Sigma^i T^j P$, 
which is determined (up to homotopy) by the degree $i$ map 
 $G(i): P^{(-i)}\to T^j P^{(0)}\to  T^j A_0$.

Now consider 
the action of $\mu_E$ on 
$g \in  \Hom_{\mc{D}_{\epsilon}(A)}(P,  \Sigma^i T^j P)$ as described by
Theorem~\ref{xxthm2.7}.  In degree $(i,j)$, the formula is 
\[
g \longmapsto (-1)^{(d+1)i} \, \Sigma^i T^j(\phi)^{-1} \circ \Phi(g) \circ \phi.
\]
 Thus we only 
need to calculate the degree $i$ map in the morphism of complexes 
$\mu_E(g)$.  Here, $\phi^{(-i)}: P^{(-i)} \to  \Phi(P)^{(-i)}$ is the map 
of position $-i$ in the isomorphism of complexes $\phi: P \to {}^{\mu_A} P$ 
fixed above.  The second map $\Phi(g)$ has the same underlying 
function as $g$, and we only consider the position $-i$, which is the map
$P^{(-i)}\to T^j P^{(0)}$. The third map $(\Sigma^i T^{j}(\phi))^{-1}$ 
when restricted to $T^j P^{(0)}$ is given by the isomorphism 
$\mu_A^{-1}: T^j ({}^{\mu_A} P^{(0)})\to T^j (P^{(0)})$.   
Now descending elements in $\Hom_A(P^{(-i)}, T^j P^{(0)})$ to 
$\Hom_{A}(P^{(-i)}, T^j A_0)$, the third map becomes the isomorphism  
$\mu^{-1} \colon T^j ({}^{\mu_A} A_0) \to T^j (A_0)$. Altogether, under the identification of 
$\Hom_{\D_{\epsilon}(A)}(P, \Sigma^iT^jP)$ with $\Hom_{A}(P^{(-i)}, T^j A_0)$ 
given above, we see that  $\mu_E(g)$ is given in complex degree $i$ 
precisely by $(-1)^{(d+1)i} \mu_A^{-1} \circ g \circ \phi^{(-i)}$.
In other words, this is the exactly the same as the 
action $f_{\mu_A}$ of $\mu_A$ on $\Ext^i_A(A_0, A_0)$ described above, except for the additional sign.

(2)  Recall our earlier calculation that $\phi^{(-1)}: P^{(-1)} \to P^{(-1)}$ is induced by applying $A \otimes_{A_0} - $ to the map 
$\overline{\sigma}: V^{(-1)} \to V^{(-1)}$, which when $A$ is generated in degrees $0$ and $1$
is the same as $\mu_A \vert_{A_1}$ under the 
identification of $V^{(-1)}$ with $A_1$.
The statement follows from restricting part (1) to degree $1$, once we identify 
$\Hom_A(P^{(-1)}, A_0)$ with $(A_1)^{\vee} = \Hom_{A_0}(A_1, A_0)$.

(3) The first statement is a special case of (2), since $\mu^{-1}: A_0 \to A_0$ is trivial when $A_0 = k$, and 
$(\mu_A \vert_{A_1})^*$ is by definition equal to $g \mapsto g \circ \mu_A \vert_{A_1}$ for 
$g \in (A_1)^* = \Hom_k(A_1, k)$.  Because $A$ and $E$ are connected, the only inner automorphism of
either algebra is the trivial one. Thus both $\mu_A$ and $\mu_E$ are uniquely determined in this case.
Note that $E$ is graded-symmetric if and only if $\mu_E = \xi_{(-1)^{d+1}, 1}$.  If $\mu_A = 1$, 
then $f_{\mu_A} = 1$ and so by part (1) we certainly get $\mu_E = \xi_{(-1)^{d+1}, 1}$.  Conversely, if 
$\mu_E = \xi_{(-1)^{d+1}, 1}$ then we get $(\mu_A \vert_{A_1})^* = 1$, and thus $\mu_A \vert_{A_1} = 1$.
Since $A$ is generated in degree $1$, $\mu_A = 1$.
\end{proof}

Part (2) of the theorem recovers and generalizes known results 
about the action of $\mu_E$ in degree $1$, which were proved in the connected 
$N$-Koszul case only in~\cite{BM}.   Note that our result allows one to calculate 
$\mu_E$ in fact in any degree, if one can calculate an explicit minimal resolution 
$P$ of $A_0$ and an explicit $\mu_A$-linear isomorphism $\phi: P \to P$.  

We note that a generalized 
AS regular algebra $A$ will be a \emph{skew Calabi-Yau} algebra as studied in \cite[Definition 0.1]{RRZ}, 
in cases where $A$ is homologically smooth in the sense of that definition 
(this is automatic, for instance, when $A$ is connected graded \cite[Lemma 1.2]{RRZ}).  
When $A$ is skew Calabi-Yau, part (1) of the theorem shows that $A$ is \emph{Calabi-Yau} ($\mu_A = 1$) if and only if 
$E$ is graded-symmetric.

An immediate consequence is the following theorem 
which answers a conjecture in \cite[Remark 4.2]{CWZ} 
affirmatively.  See also the discussion after \cite[Theorem 0.1]{CWZ}.
The definitions of the Hopf-theoretic notions involved in the following result may 
also be found in \cite{CWZ}.

\begin{theorem}
\label{xxthm4.3}
Let $A$ be a noetherian connected graded AS regular algebra 
generated in degree one with Nakayama automorphism $\mu_A$.
Let $K$ be a Hopf algebra with bijective antipode $S$ 
coacting on $A$ from the right. Suppose that the homological 
codeterminant \cite[Definition 1.5(b)]{CWZ} of the $K$-coaction 
on $A$ is the element $D\in K$ and that the $K$-coaction on 
$A$ is inner-faithful. Then
$$\eta_{D}\circ S^2=\eta_{{\mu_A}^{\tau}}\; ,$$
where $\eta_D$ is the automorphism of $K$ defined by $\eta_D(a)=D^{-1}a D$ 
and $\eta_{{\mu_A}^{\tau}}$ is the automorphism of $K$ given 
by conjugating by the transpose of the corresponding matrix
of $\mu_A$.
\end{theorem}
\begin{proof}[Proof of Theorem \ref{xxthm4.3}]
Replacing \cite[Lemma 4.1]{CWZ} by Theorem \ref{xxthm4.2}(2), 
the proof of \cite[Theorem 0.1]{CWZ} can be copied without 
much change. 
\end{proof}

The following further consequence generalizes
\cite[Theorem 0.6]{CWZ} to the non-$N$-Koszul case. The proof
is given in \cite[Proof of Theorem 0.6]{CWZ} by using Theorem \ref{xxthm4.3}. 

\begin{corollary}
\label{xxcor4.4}
Let $A$ be as in Theorem \ref{xxthm4.3} and suppose that $\rm{char}\; k = 0$. 
Further assume that $\mu_A = \xi_r$ for some $r \in k^\times$ and 
that $H$ is a finite dimensional Hopf algebra that acts on $A$ satisfying
\cite[Hypothesis 0.3]{CWZ}.  If the $H$-action on $A$ has trivial homological determinant, 
then $H$ is semisimple.
\end{corollary}

\section{The $\epsilon$-condition}
\label{xxsec5}
Let $A$ be a ${\mathbb Z}^w$-graded algebra for some $w \geq 1$.  
We say that $A$ is \emph{$\mb{N}^w$-graded} if 
$A_{i_1, i_2, \dots, i_w} \neq 0$ implies that $i_j \geq 0$ 
for all $1 \leq j \leq w$.  
The algebra $A$ is \emph{connected} if $A_{0, 0, \dots, 0} = k$. 
We also allow by convention $w = 0$, in which case $A$ is ungraded.  
We do not require that $A$ be locally finite in general in this section. 
Let $A \lGr$ be the category 
of $\mb{Z}^w$-graded left $A$-modules.  We write this as 
$(A, \mb{Z}^w) \lGr$ if we need to emphasize the grading group, for 
example if $A$ has multiple gradings.

\begin{definition}
\label{xxdef5.1} 
Let $A$ be a $\mb{Z}^w$-graded algebra, and let 
$\mc{D}^b(A \lGr)$ be the bounded derived category of $\mb{Z}^w$-graded 
left $A$-modules.  
\begin{enumerate}
\item
An object $X \in \mc{D}^b(A \lGr)$ is called an {\it $\epsilon$-object} 
if it is a perfect complex and $H^0(X[n])$ is a finite dimensional 
graded left $A$-module for all $n \in \mb{Z}$.
\item
Recall that $\mc{D}_{\epsilon}(A)$ is the subcategory of 
$\mc{D}^b(A \lGr)$ consisting of all $\epsilon$-objects.  
We say $A$ satisfies the {\it $\epsilon$-condition} if 
$\mc{D}_{\epsilon}(A) \neq 0$. We say that $(A, \mb{Z}^w)$ 
satisfies the $\epsilon$-condition if the grading group needs emphasis.
\end{enumerate}
\end{definition}

The $\epsilon$-condition is a rather mild condition and we conjecture
that all reasonable noetherian AS Gorenstein algebras satisfy it.
In the next result, we give some of its important properties.  
First, we briefly recall the notion of a twist of a graded algebra.  
Let $A$ be $\mb{Z}^w$-graded, and let a sequence 
$\wt{\sigma}:= \{\sigma_1,\cdots, \sigma_w\}\subset \Autw(A)$ of
pairwise commuting ${\mathbb Z}^w$-graded automorphisms of $A$ be given.
We write $\sigma^{r} = \sigma_1^{r_1} \dots \sigma_w^{r_w}$, for 
$r = (r_1, \dots, r_w) \in \mb{Z}^w$. 
The (left) graded twist of $A$
associated to $\tilde{\sigma}$ is a new graded algebra, denoted by
$A^{\tilde{\sigma}}$, such that $A^{\tilde{\sigma}}=A$ as
a ${\mathbb Z}^w$-graded vector space, and where the new multiplication
$\star$ of $A^{\tilde{\sigma}}$ is given by
$a \star b =\sigma^{|b|}(a) b$ for all homogeneous elements $a,b \in A$.  
Similarly, given a $\mb{Z}^w$-graded left $A$-module $M$, we may define a
$\mb{Z}^w$-graded $A^{\tilde{\sigma}}$-module $M^{\tilde{\sigma}}$ with the same 
graded vector space as $M$, but new action $a \star m = \sigma^{|m|}(a) m$ for homogeneous $a \in A^{\tilde{\sigma}}$, $m \in M$.
The functor $A \lGr \to (A^{\tilde{\sigma}}) \lGr$ which sends $M$ to $M^{\tilde{\sigma}}$ and 
is the identity map on morphisms is an equivalence of categories \cite[Theorem 3.1]{Zh}.

Below, we say that a commutative $k$-algebra is \emph{affine} if it is finitely generated as a $k$-algebra.
\begin{lemma}
\label{xxlem5.2} 
\begin{enumerate}
\item
Suppose that $A$ is a noetherian $\mb{Z}^w$-graded algebra with finite global dimension.  If there is a nonzero 
finite dimensional graded module over $A$, then $A$ satisfies the 
$\epsilon$-condition. As a consequence, every connected graded noetherian AS 
regular algebra satisfies the $\epsilon$-condition.  
\item
Let $A$ be $\mb{Z}^w$-graded.  Suppose that $\phi: {\mathbb Z}^w\to {\mathbb Z}^n$ is a group homomorphism.
Then $A$ is also $\mb{Z}^n$-graded, where $A_{r} = \bigoplus_{\{s | \phi(s) = r\}} A_s$ for 
$r \in \mb{Z}^n$. If $(A, \mb{Z}^w)$  satisfies the $\epsilon$-condition, 
then $(A, \mb{Z}^n)$  satisfies the $\epsilon$-condition.
\item
Suppose that $A$ is connected $\mb{N}^w$-graded with $w > 0$ and that $A$ 
satisfies the $\epsilon$-condition. 
If there is a $\mb{Z}^w$-graded algebra map $A\to B$ such that $B_A$ 
is finitely generated, then $(B, \mb{Z}^w)$ satisfies the 
$\epsilon$-condition. As a consequence, every affine connected graded 
commutative algebra satisfies the $\epsilon$-condition. 
\item
If $B$ is locally finite ${\mathbb N}$-graded and finite
over its affine center, then $B$ satisfies  the $\epsilon$-condition.
\item
Let $A$ and $B$ be connected $\mb{N}$-graded, and suppose that $A \otimes_k B$ 
is also noetherian. Then the following are equivalent:  
\begin{enumerate}
\item[{\rm{(i)}}] 
$(A\otimes_k B, \mb{Z}^2)$ satisfies the $\epsilon$-condition; 
\item[{\rm{(ii)}}] 
$(A \otimes_k B, \mb{Z})$ satisfies the $\epsilon$-condition; 
\item[{\rm{(iii)}}] 
both $(A, \mb{Z})$ and $(B, \mb{Z})$ satisfy the $\epsilon$-condition. 
\end{enumerate}
As a consequence, $A$ satisfies the $\epsilon$-condition if and only 
if $A[t]$ does (as either a $\mb{Z}^2$ or $\mb{Z}$-graded algebra).
\item
$A$ satisfies the $\epsilon$-condition if and only if a graded twist 
$A^{\tilde{\sigma}}$ does.
\end{enumerate}
\end{lemma}

\begin{proof} 
(1,2) These are clear.

(3) Let $X$ be a nonzero $\epsilon$-object over $A$. Let $Y=B\otimes^L_A X$. 
We claim that $Y$ is a nonzero $\epsilon$-object over $B$. To see 
that $Y$ is nonzero in the derived category, we note since $X$ is perfect 
that $Y = B \otimes_A X$ as a tensor product of complexes. 
Then $H^i(Y)=B\otimes_A H^i(X)$, where $i$ is the integer such that 
$H^i(X)\neq 0$ and $H^j(X)=0$ for all $j>i$.   Let $\mf{m}$ be the unique graded maximal 
ideal of $A$, and let $M = A/\mf{m} = A_{(0,0, \dots, 0)}$ be the unique simple graded $A$-module.   Since $H^i(X)\neq 0$, 
there is a surjective $A$-module map $f:H^i(X) \to M$.    Since $B$ is a nonzero finitely generated 
graded right $A$-module, the graded Nakayama's Lemma implies that 
$B\otimes_A M\neq 0$.   Thus $B\otimes_A H^i(X)\neq 0$ and therefore $H^i(Y)\neq 0$. 

Next, since $X$ is a perfect complex of $A$-modules, $Y=B\otimes_A X$ 
is a perfect complex of $B$-modules. Last, to see that $Y$ is an 
$\epsilon$-object over $B$, we note that 
$H^0(Y[n])=H^n(Y)=\Tor^A_{-n}(B,X)$. So it suffices to show the 
assertion that $\Tor^A_{i}(B,X)$ is finite dimensional for all $i$. 
By exact sequences and induction on $\sum_n \dim_k H^0(X[n])$ (forgetting 
that $X$ is perfect), we may reduce to the case that $X=M$, where $M$ is the 
unique simple graded left $A$-module.   Since $A$ is noetherian, and 
$B_A$ is finitely generated, we may replace $B$ in $\mc{D}^{-}(A \lgr)$ 
by a graded projective resolution $P$ of $B_A$ such that each term in the 
projective resolution is finitely generated.  Then $Y = P \otimes_A M$ 
has finite dimensional cohomology as required.  

For the consequence, take $A$ to be a connected graded affine commutative
polynomial ring. By part (1), $A$ satisfies the $\epsilon$-condition.
By the above paragraph, every factor ring of $A$ satisfies the 
$\epsilon$-condition. 

(4) Let $Z$ be the center of $B$ and let $A=k\oplus Z_{\geq 1}$.  Since $Z$ is affine 
by assumption, it easily follows that $A$ is affine as well.
By part (3), $A$ satisfies the $\epsilon$-condition. Since $B$ is finite
over $Z$, and $Z$ is clearly finite over $A$, $B$ is finite over $A$. The assertion follows from
part (3). 

(5)  The implication (i) $\Rightarrow$ (ii) is immediate from part (2).  For 
(ii) $\Rightarrow$ (iii), using the $\mb{Z}$-graded homomorphism 
$A \otimes_k B \to A \otimes_k B/B_{\geq 1} \cong A \otimes_k k \cong A$, 
$A$ is a finitely generated $A \otimes_k B$-module.  Thus 
$(A, \mb{Z})$-satisfies the $\epsilon$-condition by part (3), and 
similarly for $(B, \mb{Z})$.  Now suppose that $(A, \mb{Z})$ and 
$(B, \mb{Z})$ satisfy the $\epsilon$-condition, and let $X$ and $Y$ be nonzero $\epsilon$-objects over $A$ 
and $B$ respectively. Then $X\otimes_k Y$ is a nonzero 
$\mb{Z}^2$-graded $\epsilon$-object over $A\otimes_k B$, proving 
(iii) $\Rightarrow$ (i). The consequence follows by taking $B=k[t]$.

(6) Note that the equivalence between $(A, \mb{Z}^w) \lGr$ and 
$(A^{\tilde{\sigma}}, \mb{Z}^w) \lGr$ is naturally extended to the
derived level. Then one checks that a nonzero object in 
$\D_{\epsilon}(A)$ maps to a nonzero object in 
$\D_{\epsilon}(A^{\tilde{\sigma}})$ under this equivalence.
\end{proof}

We can now prove the final major theorem of the paper.

\begin{theorem}
\label{xxthm5.3}
Let $A$ be connected graded noetherian AS Gorenstein.  Suppose 
that $\D_{\epsilon}(A) \neq 0$. Then $\hdet \mu_A = 1$.
As a consequence, if $A$ is connected graded noetherian 
AS regular, then $\hdet \mu_A= 1$.
\end{theorem}

\begin{proof}
The proof depends heavily on the proof of \cite[Lemma 6.2]{RRZ}.
The basic idea is to find a closely related AS Gorenstein 
algebra $B$ such that $\hdet \mu_B = \hdet \mu_A$ and 
$\mu_B = \xi_c$ for some $c$, and then apply Theorem~\ref{xxthm4.1}.  

In more detail, following the proof of \cite[Lemma 6.2]{RRZ}, first 
one replaces $A$ by a polynomial extension $A[s]$ if necessary to avoid 
the case of AS index $-1$.  Note that the hypothesis 
$\D_{\epsilon}(A) \neq 0$ still holds, by Lemma~\ref{xxlem5.2}(5).

Then one takes the algebra $A[t]$, as a $\mb{Z}^2$-graded algebra, 
and does a $\mb{Z}^2$-graded twist. Note that $(A[t], \mb{Z}^2)$ 
satisfies the $\epsilon$-condition by Lemma~\ref{xxlem5.2}(5), 
and so will a $\mb{Z}^2$-graded twist by Lemma~\ref{xxlem5.2}(6).  
Then one considers $A[t]$ as a $\mb{Z}$-graded algebra via the 
``total degree" homomorphism $\mb{Z}^2 \to \mb{Z}$; $(A[t], \mb{Z})$ 
satisfies the $\epsilon$-condition still by Lemma~\ref{xxlem5.2}(2).
Finally, one does a $\mb{Z}$-graded twist, arriving at an algebra $B$, 
where $(B, \mb{Z})$ satisfies the $\epsilon$-condition by 
Lemma~\ref{xxlem5.2}(6).  By the proof of \cite[Lemma 6.2]{RRZ}, 
$\mu_B = \xi_c$  and $\hdet \mu_B = \hdet \mu_A$.

Now by Theorem~\ref{xxthm4.1}(2), $\hdet \mu_B = c^{{\bfl}} = 1$, 
where ${\bfl}$ is the AS index of $B$.  So $\hdet \mu_A = 1$.

The consequence follows from Lemma \ref{xxlem5.2}(1).
\end{proof}

\begin{corollary}
\label{xxcor5.4}
If $A$ is connected $\mb{N}$-graded noetherian AS Gorenstein and has any of 
the following additional properties, then $\hdet \mu_A = 1$:  
\begin{enumerate}
\item[{\rm{(i)}}] 
$A$ is a graded twist of an algebra which is finite over its affine center; or 
\item[{\rm{(ii)}}] 
$A$ is a surjective image of a noetherian AS regular algebra.
\end{enumerate}
\end{corollary}

\begin{proof}
By Theorem \ref{xxthm5.3}, we just need to verify that $A$ satisfies 
the $\epsilon$-condition.   This holds by Lemma \ref{xxlem5.2}(4)(6) 
in case (i), and by Lemma~\ref{xxlem5.2}(1)(3) in case (ii).
\end{proof}

The following conjecture was made in \cite[Conjecture 6.4]{RRZ}.
\begin{conjecture}\label{xxcon5.5} Let $A$ be a noetherian connected graded 
AS Gorenstein algebra. 
Then $\hdet(\mu_A)=1$. 
\end{conjecture}

Most known AS Gorenstein algebras satisfy one of the 
conditions in Corollary~\ref{xxcor5.4}, and so 
we have now proved the conjecture for the most common kinds of examples.
Of course, the conjecture would be completely solved if there were a positive 
answer to the following question.

\begin{question}
\label{xxque5.6}
Let $A$ be a noetherian connected graded AS Gorenstein algebra.
Does $A$ satisfy the $\epsilon$-condition?
\end{question}

\providecommand{\bysame}{\leavevmode\hbox to3em{\hrulefill}\thinspace}
\providecommand{\MR}{\relax\ifhmode\unskip\space\fi MR }
\providecommand{\MRhref}[2]{%
  \href{http://www.ams.org/mathscinet-getitem?mr=#1}{#2}
}
\providecommand{\href}[2]{#2}


\end{document}